\documentclass[12pt]{amsart}
\usepackage[latin1]{inputenc}
\usepackage[english]{babel}
\usepackage[T1]{fontenc}
\usepackage{amssymb,amsmath,amsthm,amsfonts,enumerate}
\usepackage[autostyle]{csquotes}
\usepackage{alphalph,cite}
\usepackage{nameref}
\usepackage{verbatim}
\usepackage{mathrsfs}
\usepackage{manfnt}
\usepackage{pifont}
\usepackage{yfonts}
\usepackage[perpage]{footmisc}
\usepackage{enumitem}
\usepackage[all]{xy} 
\usepackage{wasysym}

\usepackage{stmaryrd}

\usepackage{apptools}
\usepackage{stackengine}
\usepackage{scalerel}
\usepackage{array}
\usepackage{datetime}
\usepackage{xparse} 
\usepackage{tikz-cd}
\usepackage{rotating}
\usepackage{esvect}
\usepackage[colorlinks=true, linkcolor={blue}, citecolor={magenta}, urlcolor={green}]{hyperref}
\usepackage{mathtools}
\usepackage{sansmath}

\makeatletter
\@namedef{subjclassname@2010}{ \textup{2010} Mathematics Subject Classification}
\makeatother
\theoremstyle{definition}
\newtheorem{theorem}{Theorem}[subsection]

\newtheorem{theoremletters}{Theorem}

\renewcommand{\thefootnote}{\arabic{footnote}}

\theoremstyle{definition}
\newtheorem{lemma}[theorem]{Lemma}
\theoremstyle{definition}

\theoremstyle{definition}

\theoremstyle{definition}

\theoremstyle{definition}
\newtheorem{remark}[theorem]{Remark}
\theoremstyle{definition}
\newtheorem{example}[theorem]{Example}
\theoremstyle{definition}

\theoremstyle{definition}
\newtheorem{definition}[theorem]{Definition}
\numberwithin{equation}{section}

\numberwithin{equation}{section}
\theoremstyle{definition}

\theoremstyle{definition}
\newtheorem{notation}[theorem]{Notation}
\theoremstyle{definition}

\theoremstyle{definition}

\theoremstyle{definition}

\theoremstyle{definition}

\theoremstyle{definition}
\newtheorem{problem}[theorem]{Problem}

\newcommand{\norm}[1]{\left\lVert#1\right\rVert}
\newcommand{\vertiii}[1]{{\left\vert\kern-0.25ex\left\vert\kern-0.25ex\left\vert #1 
		\right\vert\kern-0.25ex\right\vert\kern-0.25ex\right\vert}}
\newcommand{\abs}[1]{\left\lvert#1\right\rvert}

\newcommand{\nocontentsline}[3]{}
\newcommand{\tocless}[2]{\bgroup\let\addcontentsline=\nocontentsline#1{#2}\egroup}

\renewcommand{\thefootnote}{\alph{footnote}}

\makeatletter
\def\author@andify{
	\nxandlist {\unskip ,\penalty-1 \space\ignorespaces}
	{\unskip {} \@@and~}
	{\unskip \penalty-2 \space \@@and~}
}
\makeatother

\makeatletter
\def\l@subsection{\@tocline{2}{0pt}{1pc}{5pc}{}} \def\l@subsection{\@tocline{2}{0pt}{2pc}{6pc}{}}
\makeatother

\begin{document}
	\title{A few last words on pointwise multipliers of Calder{\' o}n--Lozanovski{\u \i} spaces}
	
	\author{Tomasz Kiwerski}
	\address[Tomasz Kiwerski]{Pozna\'{n} University of Technology, Institute of Mathematics, Piotrowo 3A, 60-965 Pozna\'{n}, Poland}
	\email{\href{mailto:tomasz.kiwerski@gmail.com}{\tt tomasz.kiwerski@gmail.com}}
	
	\author{Jakub Tomaszewski}
	\address[Jakub Tomaszewski]{Pozna\'{n} University of Technology, Institute of Mathematics, Piotrowo 3A, 60-965 Pozna\'{n}, Poland}
	\email{\href{mailto:tomaszewskijakub@protonmail.com}{\tt tomaszewskijakub@protonmail.com}}
	
	\maketitle
	
	\begin{abstract}
		We will provide a complete description of the space $M(X_F,X_G)$ of pointwise multipliers between two Calder{\' o}n--Lozanovski{\u \i}
		spaces $X_F$ and $X_G$ built upon a rearrangement invariant space $X$ and two Young functions $F$ and $G$.
		Meeting natural expectations, the space $M(X_F,X_G)$ turns out to be another Calder{\' o}n--Lozanovski{\u \i} space $X_{G \ominus F}$
		with $G \ominus F$ being the appropriately understood generalized Young conjugate of $G$ with respect to $F$.
		Nevertheless, our argument is not a mere transplantation of existing techniques and requires a rather delicate analysis of the interplay
		between the space $X$ and functions $F$ and $G$.
		Furthermore, as an example to illustrate applications, we will solve the factorization problem for Calder{\' o}n--Lozanovski{\u \i} spaces.
		All this not only complements and improves earlier results (basically giving them the final touch), but also confirms the conjecture formulated
		by Kolwicz, Le{\' s}nik and Maligranda in
		[{\it Pointwise multipliers of Calder{\' o}n--Lozanovski{\u \i} spaces}, Math. Nachr. {\bf 286} (2012), no. 8-9, 876--907].
		We will close this work by formulating a number of open questions that outline a promising panorama for future research.
	\end{abstract}
	
	\tableofcontents
	
	\renewcommand{\thefootnote}{\fnsymbol{footnote}} \footnotetext{
		\textit{Date:} \longdate{\today}.
		
		
	}
	
	\section{{\bf Introduction}} \label{SECTION : Introduction}
	
	\noindent
	1.a.\,\,{\bf Goals.}
	The main goal of this paper is to provide a description of the space $M(X_F,X_G)$ of pointwise multipliers between
	two Calder{\' o}n--Lozanovki{\u \i} spaces $X_F$ and $X_G$.
	By this we simply mean that we are looking for conditions on the function $f$ guaranteeing that the induced multiplication
	operator $M_f \colon g \mapsto fg$ is bounded when acting from $X_F$ into $X_G$. 
	Here, and hereinafter, $X$ is a rearrangement invariant space, while $F$ and $G$ are two Young functions.
	
	Anyway, before presenting our results in more detail, let us first try to explain what the sources of this problem are,
	what was our motivation and why all this may be even important.
	(For all unexplained concepts that have already appeared, or will appear, we refer to Sections~\ref{SECTION : Toolbox},
	\ref{SUBSECTION : Young conjugate} and \ref{SUBSECTION : Factorization}.)
	\newline
	
	\noindent
	1.b.\,\,{\bf Back to Orlicz spaces.}
	From a historical perspective, the {\it fons et origo} of the problem of describing the space of pointwise multipliers $M(X,Y)$
	between two function spaces $X$ and $Y$ can most likely be traced back to O'Neil's work from 1965.
	In \cite[Problem~6.2]{ONe65}, he asks
	\newline
	
	\noindent
	{\bf Problem.} (O'Neil, 1965).
	{\it How, given two Young functions $F$ and $G$, to choose the third Young function, say $H$, so that the space $M(L_F,L_G)$
	of pointwise multipliers between two Orlicz spaces $L_F$ and $L_G$ is another Orlicz space $L_H$?}
	\newline
	
	Many authors over nearly half a century have obtained mainly partial results in this direction
	(we refer, for example, to \cite{And60}, \cite{Mau74}, \cite{ONe65}, \cite{MP89}, \cite{MN10} and \cite{ZR67}).
	Putting some technical details aside, the generic result of this type looks more-or-less like this
	\newline
	
	\noindent
	{\bf Theorem.} (Maligranda and Nakaii, 2010)
	{\it If three given Young functions $F$, $G$ and $H$ satisfy the relation $F^{-1} G^{-1} \approx H^{-1}$, then $M(L_F,L_G) = L_H$.}
	\newline
	
	Unfortunately, such a result do not lead to any explicit formula for the Young function $H$ generating the space of pointwise multipliers.
	Worse, there are pairs of Young functions, say $F$ and $G$, for which there is no third Young function $H$ satisfying the relation
	$F^{-1} G^{-1} \approx H^{-1}$ even though the space $M(L_F,L_G)$ is a non-trivial Orlicz space.
	There are probably two noted examples of this kind in the existing literature, namely, \cite[Example~7.8]{KLM12} and \cite[Example~A.2]{KT22}.
	The former one concerns Orlicz spaces defined on $\mathbb{I} = [0,1]$ and gives $M(L_F,L_G) = L_{\infty}[0,1]$,
	while the second one covers both remaining cases, that is, $\mathbb{I} = [0,\infty)$\footnote{Actually, what is written in \cite[Example~A.2]{KT22}
	does not cover the case $\mathbb{I} = [0,\infty)$ at all, but with some modicum of solid effort can be modified to work then too.}
	and $\mathbb{I} = \mathbb{N}$, but is also somehow more sophisticated, because $M(L_F,L_G) \neq L_{\infty}$.
	
	Some satisfactory results crowning these struggles have only been obtained much more recently in \cite{DR00} and \cite{LT17}.
	The most important chisel that allows to crack O'Niels problem is the construction of the function $G \ominus F$.
	Recall that the value $(G \ominus F)(t)$ for $t \geqslant 0$ is defined to be $\sup \left\{ G(st) - F(s) \right\}$,
	where the supremum is taken over all $s \geqslant 0$ for which $F(s)$ is finite (cf. Definition~\ref{DEF : generalized Young conjugate};
	see Section~\ref{SUBSECTION : Young conjugate} for a discussion why this construction is natural).
	With the function $G \ominus F$ at our disposal, the mentioned solution looks as follows
	\newline
	
	\noindent
	{\bf Theorem.} (Djakov and Ramanujan, 2010; Le{\' s}nik and Tomaszewski, 2017)
	{\it Let $F$ and $G$ be two Young functions.
	Then $M(L_F,L_G) = L_{G \ominus F}$.}
	\newline
	
	For the reasons outlined above, our work can naturally be placed within the framework of O'Neil's problem in which the classical Orlicz spaces $L_F$ and $L_G$
	are replaced by a much more general\footnote{After all, the class of Orlicz spaces $L_F$ is just a particular case of Calder{\' o}n--Lozanovski{\u \i}'s
	construction $X_F$, where the space $X$ is chosen to be $L_1$.} constructions of the Calder{\' o}n--Lozanovski{\u \i} spaces $X_F$ and $X_G$, respectively.
	\newline
	
	\noindent
	1.c.\,\,{\bf One conjecture.}
	Apart from these historical marginalia, our work starts where \cite{KLM12} and \cite{KLM14} end.
	In fact, our initial motivation comes from the following conjecture formulated at the end of \cite{KLM12}.
	\newline
	
	\noindent
	{\bf Conjecture.} (Kolwicz, Le{\' s}nik and Maligranda, 2012).
	{\it Let $X$ be a rearrangement invariant function space defined on $\mathbb{I} = [0,1]$.
	Suppose that $X \neq L_{\infty}[0,1]$.
	Then we have the equality $M(X_F,X_G) = X_{G \ominus F}$.}
	\newline
	
	While working on this problem we realize, to our pleasant surprise, that we were not only able to confirm the above conjecture,
	but even provide a necessary and sufficient conditions on the triple $(X,F,G)$ for equality $M(X_F,X_G) = X_{G \ominus F}$ to holds
	(this, we hope, justifies the slightly funeral character of the title of our work).
	
	Notabene, we were also able to resolve the case when the space $X$ is defined on $\mathbb{I} = [0,\infty)$.
	This is also a bit of surprise, because the structure of rearrangement invariant spaces defined on $\mathbb{I} = [0,\infty)$
	is much more complicated that on $\mathbb{I} = [0,1]$.
	For instance, there is a whole menagerie of spaces that resemble $L_{\infty}[0,\infty)$ in the sense that they have no order continuous
	functions (they are, so to speak, \enquote{extremely} non-separable).
	In fact, just take any rearrangement invariant space defined on $\mathbb{I} = [0,\infty)$, say $X$, and consider the intersection
	space $X \cap L_{\infty}$ defined via the norm $\norm{f}_{X \cap L_{\infty}} = \max\{ \norm{f}_X, \norm{f}_{L_{\infty}} \}$.
	On the other hand, up to the equivalence of norms, there is only one rearrangement invariant space defined on $\mathbb{I} = [0,1]$
	devoid of order continuous functions, namely, $L_{\infty}[0,1]$ itself.
	This is most likely the reason why the above conjecture was formulated only in the case when the space $X$ is defined on $\mathbb{I} = [0,1]$.
	\newline
	
	\noindent
	1.d.\,\,{\bf Overview.}
	Let us take a closer look at our results.
	The first one concerns the description of the space of pointwise multipliers.
	
	\begin{theoremletters}[Theorem~\ref{THEOREM : PM between CL} and Theorem~\ref{THEOREM : main thm not nice};
		see also Notation~\ref{NOTATION : Nice triple} for what \enquote{nice triple} means] \label{THEOREM A}
		\hfill
		\begin{itemize}
			\item[(1)] {\it $M(X_F,X_G) = X_{G \ominus F}$ if, and only if, the triple $(X,F,G)$ is nice.}
			\item[(2)] {\it $M(X_F,X_G) = X_{G \ominus_1 F}$ if, and only if, either the triple $(X,F,G)$ fails to be nice or the space $X$ is a sequence space.}
		\end{itemize}
	\end{theoremletters}

	In particular, since triples $(X,F,G)$, where $X$ is a rearrangement invariant space on $\mathbb{I} = [0,1]$ such that $X \neq L_{\infty}[0,1]$,
	are always nice, so (1) corresponds exactly with Kolwicz, Le{\' s}nik and Maligranda's conjecture (cf. Remark~\ref{REMARK : not nice on [0,1]}).
	
	Inspired by \cite{KLM14}, we will use these results to provide a final picture on the factorization of Calder{\' o}n--Lozanovski{\u \i} spaces.
	This is, in some sense, one of the most natural applications of knowledge about the space of pointwise multipliers.
	The idea behind the factorization problem for Calder{\' o}n--Lozanovski{\u \i} spaces (which goes back to Lozanovski{\u \i}'s classical result;
	see Section~\ref{SUBSECTION : Factorization} for more details) is to ask under what assumptions each multiplication operator $M_f$ acting from
	$L_{\infty}$ into $X_G$ admits a factorization
	\begin{equation*}
		\begin{tikzcd}
			L_{\infty} \arrow[rd, "M_g"', dashed] \arrow[rr, "M_f"] &                      & X_G \\
			& X_F \arrow[ru, "M_h"', dashed] &  
		\end{tikzcd}
	\end{equation*}
	(which, of course, means that $M_f = M_{gh} = M_g \circ M_h$) with
	\begin{equation*}
		\lVert M_g \colon L_{\infty} \rightarrow X_F \rVert \, \lVert M_h \colon X_F \rightarrow X_G \rVert \approx \lVert M_f \colon L_{\infty} \rightarrow X_G \rVert.
	\end{equation*}
	Equivalently, after translating factorization into a more popular dialect of \enquote{spaces}, we ask whether
	\begin{equation} \label{EQ : factorization more complicated}
		M(L_{\infty},X_F) \odot M(X_F,X_G) = M(L_{\infty},X_G).
	\end{equation}
	Here, by $M(L_{\infty},X_F) \odot M(X_F,X_G)$ we mean the pointwise product space, that is, a vector space of all measurable functions $f$
	that can be written as a pointwise product $gh$ with $g \in M(L_{\infty},X_F)$ and $h \in M(X_F,X_G)$.
	When endowed with the functional
	\begin{equation*}
		f \mapsto \inf \{ \lVert M_g \colon L_{\infty} \rightarrow X_F \rVert \, \lVert M_h \colon X_F \rightarrow X_G \rVert \},
	\end{equation*}
	where the infimum is taken over all decompositions $f = gh$, it becomes a quasi-Banach function space (see Section~\ref{SUBSECTION : Products} for further discussion).
	Actually, once we realize that $M(L_{\infty},X_F) = X_F$ and $M(L_{\infty},X_G) = X_G$, we can rewrite \eqref{EQ : factorization more complicated}
	in a simpler form
	\begin{equation}
		X_F \odot M(X_F,X_G) = X_G.
	\end{equation}
	This is the quality that is usually meant when talking about factorization of function spaces.
	Our overall result in this direction is as follows.

	\begin{theoremletters}[Theorem~\ref{THEOREM : Factorization of CL}] \label{THEOREM B}
		{\it Let $X$ be a rearrangement invariant space.}
		\begin{itemize}
			\item[(1)] {\it Suppose that $X \neq L_{\infty}$. Then the factorization $X_F \odot M(X_F,X_G) = X_G$ holds if,
			and only if, $F^{-1}(G \ominus F)^{-1} \approx G^{-1}$ (this equivalence should be understood in the appropriate way).}
			\item[(2)] {\it Suppose that $X = L_{\infty}$. Then $X_F \odot M(X_F,X_G) = X_G$ regardless of the functions $F$ and $G$.}
		\end{itemize}
	\end{theoremletters}

	Theorems~\ref{THEOREM A} and \ref{THEOREM B} extend and complete all previously known results in this area (see, for example,
	by Ando \cite{And60}; Dankert \cite{Dan74}; Djakov and Ramanujan \cite{DR00}; Kolwicz, Le{\' s}nik and Maligranda \cite{KLM12}, \cite{KLM14};
	Le{\' s}nik and Tomaszewski \cite{LT17}; Maligranda and Nakai \cite{MN10}; Maligranda and Persson \cite{MP89}; O'Neil \cite{ONe65};
	Zabre{\u \i}ko and Ruticki{\u \i} \cite{ZR67}).
	
	That aside, it seems to us that there is one more, slightly obscure, thing about Theorems~\ref{THEOREM A} and \ref{THEOREM B}
	that anyway encompasses aesthetic imperative.
	Their formulations are simply transparent, which quite ruthlessly corresponds with previous results.
	\newline
	
	\noindent
	1.e.\,\,{\bf Key ideas.}
	A few things that could be considered the most innovative in our approach should probably be clearly highlighted.
	Let's do this now.
	
	$\bigstar$ Contrary to the results so far, we are not trying to find some conditions on the Young functions $F$ and $G$
	that guarantee that the space $M(X_F,X_G)$ is again the Calder{\' o}n--Lozanovski{\u \i} space.
	Instead, we simply want to describe the space $M(X_F,X_G)$.
	Anyhow, the fact that $M(X_F,X_G)$ is indeed another Calder{\' o}n--Lozanovski{\u \i} space can be considered a fortunate coincidence.
	
	$\bigstar$ We heavily use the machinery of generalized Young's conjugate functions.
	This idea is deeply rooted in the recent work of Le{\' s}nik and the second-named author \cite{LT17} and \cite{LT21},
	which in turn draw inspiration from the pioneering paper of Djakov and Ramanujan \cite{DR00}.
	
	$\bigstar$ While working on factorization, we are not interested in the factorization of the form $X_F \odot X_H = X_G$.
	Instead, we want to know when $X_F \odot M(X_F,X_G) = X_G$.
	This is only apparently weaker.
	In fact, if $X_F \odot X_H = X_G$ then $X_H \hookrightarrow M(X_F,X_G)$.
	This simple observation, combined with the knowledge that $M(X_F,X_G) = X_{G \ominus F}$, leads to a factorization problem
	of the form $X_F \odot X_{G \ominus F} = X_G$.
	This is somehow a little easier than the initial problem.
	
	$\bigstar$ We do not need any separability assumptions imposed on the space $X$.
	This is probably quite remarkable, because this assumption has been present, in one form or another, in almost all previous results.
	This also clearly distinguishes our case from the case of Orlicz spaces (after all, $L_F = (L_1)_F$ and $L_1$ is separable).
	
	$\bigstar$ The formalism we propose supports all three separable measure spaces, namely, $\mathbb{I} = [0,1]$ or $\mathbb{I} = [0,\infty)$
	with the Lebesgue measure $m$ and $\mathbb{I} = \mathbb{N}$ with the counting measure $\#$, simultaneously.
	From the perspective of Luxemburg's representation theorem, that's all one could ever expect in the realm of rearrangement
	invariant spaces.
	\newline
	
	\noindent
	1.f.\,\,{\bf Outline.}
	Let us now briefly describe the organization of this work.
	Overall, the paper is divided into six sections.
	
	The first section after Introduction, that is, Section~\ref{SECTION : Toolbox}, is of a preliminary nature.
	Here we will recollect the needful background and provide a handful of useful facts.
	We will pay special attention to the Calder{\' o}n--Lozanovski{\u \i} construction (Section~\ref{SUBSECTION : Calderon construction}),
	without forgetting about the construction of the space of pointwise multipliers (Section~\ref{SUBSECTION : Multipliers})
	and pointwise products (Section~\ref{SUBSECTION : Products}).
	
	Next part, that is, Section~\ref{SECTION : Main results}, should be considered as the main part of this work.
	Apart from a brief discussion about the generalized Young conjugate (Section~\ref{SUBSECTION : Young conjugate}),
	it contains essentially two results, which are Theorem~\ref{THEOREM : PM between CL} and Theorem~\ref{THEOREM : main thm not nice}.
	They provide necessary and sufficient condition for the equality $M(X_F, X_G) = X_{G \ominus F}$ to hold.
	We will conclude this section with some specific examples to illustrate our general results (see Examples~\ref{EXAMPLE : rachuneczki 1},
	\ref{EXAMPLE : rachuneczki 2}, \ref{EXAMPLE : Maligranda i Persson} and \ref{EXAMPLE : MP jak skacze}). 
	
	Section~\ref{SECTION : Applications} is devoted to applications.
	After a short introduction to the factorization problem (Section~\ref{SUBSECTION : Factorization}), we present several technical lemmas that
	will prove necessary later (Section~\ref{SUBSECTION : Products of CL spaces}). 
	The culmination is Section~\ref{SUBSECTION : On factorization}, where we prove factorization theorem for Calder{\' o}n--Lozanovski{\u \i} spaces
	(precisely, see Theorem~\ref{THEOREM : Factorization of CL}).
	
	Lastly, in Section~\ref{SECTION : Open problems}, we have collected several open problems whose potential solution
	would significantly extend and complement the results obtained here.
	\newline
	
	\noindent
	1.g.\,\,{\bf Acknowledgments.}
	The results presented here confirm certain beliefs that the second-named author acquired after completing his doctoral dissertation \cite{Tom21}
	(written under supervision of Professors Karol Le{\' s}nik and Ryszard P{\l}uciennik).
	
	\section{{\bf Statements and Declarations}}
	
	\noindent
	{\bf Funding.} Our research was carried out at the Pozna{\' n} University of Technology (grant number 0213/SBAD/0120).
	\newline
	
	\noindent
	{\bf Conflict of Interest.}
	The authors declare that they have no known competing financial interest or personal relationship that could have appeared
	to influence the work reported in this paper.
	\newline
	
	\noindent
	{\bf Data Availibility.}
	All data generated or analysed during this study are included in this article.
	
	\section{{\bf Toolbox}} \label{SECTION : Toolbox}
	
	Our notation and terminology is rather standard and is in-line with what can be found in the classical monographs
	by Bennett and Sharpley \cite{BS88}, and by Lindenstrauss and Tzafriri \cite{LT79}.
	Below, we provide in detail the most important definitions, terminology and some essential facts which we will use hereinafter.
	This is dictated not only by the reader's convenience, but also by the need to organize some of the material and adapt it
	to our purposes.
	We have also made some effort to ensure that this work is largely self-contained.
	Other concepts omitted here will be introduced where necessary.
	
	\subsection{Banach ideal spaces}
	
	For a complete and $\sigma$-finite measure space $(\Omega,\Sigma,\mu)$, let $L_0(\Omega)$, briefly just $L_0$ if the context leaves no ambiguity,
	be the set consisting of all equivalence classes, by the Kolmogorov quotient modulo equality almost everywhere, of real- or complex-valued measurable
	functions defined on $\Sigma$.
	As usual, we consider the space $L_0$ with the topology of convergence in measure on sets of finite measure, that is, the topology of local convergence in measure.
	This makes $L_0$ an $F$-space.
	
	A Banach space $X$ is called a {\bf Banach ideal space} (using another common nomenclature, a {\bf K{\" o}the space} or a {\bf Banach function lattice})
	if the following two conditions hold:
	\begin{itemize}
		\item[(1)] $X$ is a linear subspace of $L_0$;
		\item[(2)] if $\abs{f(\omega)} \leqslant \abs{g(\omega)}$ almost everywhere on $\Omega$, with $f$ measurable
		and $g \in X$, then also $f \in X$ and $\norm{f} \leqslant \norm{g}$ (the so-called {\bf ideal property}).
	\end{itemize}
	
	Further, recall that a Banach ideal space $X$ is said to be {\bf rearrangement invariant} (or {\bf symmetric}) if additionally:
	\begin{itemize}
		\item[(3)] for any two measurable functions, say $f$ and $g$, with
		$\mu\left( \left\{ \omega \in \Omega \colon \abs{f(\omega)} > \lambda \right\} \right) = \mu\left( \left\{ \omega \in \Omega \colon \abs{g(\omega)} > \lambda \right\} \right)$
		for all $\lambda \geqslant 0$ and $f \in X$, it follows that $g \in X$ and $\norm{f} = \norm{g}$;
		\item[(4)] for any increasing sequence $\{ f_n \}_{n=1}^{\infty}$ of non-negative functions from $X$ converging almost
		everywhere to $f$ such that $\sup \{ \norm{f_n} \colon n \in \mathbb{N} \}$ is finite, it follows that $f$ belongs to
		$X$ and $\norm{f} = \sup \{ \norm{f_n} \colon n \in \mathbb{N} \}$
		(this is the so-called {\bf Fatou property}\footnote{Roughly speaking, this means that the closed unit ball of $X$
		is also closed with respect to the topology of local convergence in measure.}).
	\end{itemize}
	
	Using the language of the interpolation theory, rearrangement invariant spaces defined as above are exactly interpolation spaces with respect to the couple $(L_{\infty}, L_1)$.
	Evidently, condition (3) ensures that $\norm{f} = \norm{f^{\star}}$, where $f^{\star}$ is the {\bf non-increasing rearrangement} of $f$, that is,
	$f^{\star}(t) \coloneqq \inf \left\{ \lambda > 0 \colon \mu\left( \left\{ \omega \in \Omega \colon \abs{f(\omega)} > \lambda \right\} \right) \leqslant t \right\}$
	for $t \geqslant 0$.
	Moreover, thanks to {\bf Luxemburg's representation} theorem (see \cite[Theorem~4.10, p.~62]{BS88} and \cite[pp. 114--115]{LT79}),
	it is enough to consider rearrangement invariant spaces defined on one of the following three separable\footnote{It is well-known that the measure space
	$(\Omega,\Sigma,\mu)$ is separable if, and only if, $L_1(\Omega)$ is separable as a Banach space (see \cite[p.~137]{Zaa67}).} measure spaces:
	\begin{itemize}
		\item[$\bigstar$] the set of positive integers $\mathbb{I} = \mathbb{N}$ with the counting measure $\#$
		(in this situation we will talk about {\bf sequence spaces});
		\item[$\bigstar$] the unit interval $\mathbb{I} = [0,1]$ or the half-line $\mathbb{I} = [0,\infty)$ with the usual Lebesgue measure
		(with the convention that we will call them {\bf function spaces}).
	\end{itemize}
	Note that the most prolific classes of function spaces such as Lebesgue spaces, Orlicz spaces and Lorentz spaces are indeed rearrangement invariant.
	
	A function $f$ from a Banach ideal space $X$ is said to be {\bf order continuous}\footnote{Plainly, any function $f$ from $L_1(\Omega)$ is order continuous.
	Thus, this definition is nothing else, but Lebesgue's dominated convergence theorem under more abstract clothes.} if, for any sequence $\{ f_n \}_{n=1}^{\infty}$
	of positive functions that is bounded above by $\abs{f}$ and converges almost everywhere to zero, it follows that $\{ f_n \}_{n=1}^{\infty}$ is norm null sequence.
	By $X_o$ we denote a closed subspace of all order continuous functions from $X$.
	We will say that the space $X$ is {\bf order continuous} if $X = X_o$.
	Note that a Banach ideal space $X$ is order continuous if, and only if, it is separable (see \cite[Theorem~5.5, p.~27]{BS88}).
	Therefore, hereinafter, we will use these terms interchangeably.
	
	By the {\bf K{\" o}the dual} (or, the {\bf associated space}) $X^{\times}$ of a given Banach ideal space $X$ we understand here a vector space all measurable functions
	$f$ such that $fg$ is integrable for all $g \in X$ equipped with the norm $\norm{f}_{X^{\times}} \coloneqq \sup \{ \norm{fg}_{L_1} \colon \norm{g}_X \leqslant 1 \}$.
	Recall that $X \equiv X^{\times \times}$ if, and only if, the norm in $X$ has the Fatou property.
	Moreover, for any given order continuous Banach ideal space $X$, its K{\" o}the dual $X^{\times}$ can be naturally identified with the topological dual $X^{*}$,
	that is, the space of all continuous linear forms on $X$ (see \cite[Corollary~4.3, p.~23]{BS88}).
	Basically for this reason, a Banach ideal space $X$ is reflexive if, and only if, both $X$ and $X^{\times}$ are order continuous
	(see \cite[Corollary~4.4., p.~23]{BS88}).
	
	Recall that the {\bf fundamental function} $\psi_X \colon \mathbb{I} \rightarrow [0,\infty)$ of a rearrangement invariant space $X$
	is defined by the formula $\psi_X(t) \coloneqq \lVert {\bf 1}_{[0,t)} \rVert_{X}$ for $t \in \mathbb{I}$.
	We will say that the fundamental function $\psi_X$ {\bf does not vanish at zero} if $\psi_X(0^{+}) \coloneqq \lim_{t \rightarrow 0^{+}} \psi_X(t) > 0$.
	It is straightforward to see that $\psi_X$ does not vanish at zero if, and only if, $X$ is a subspace of $L_{\infty}$ and
	if, and only if, the ideal $X_o$ is trivial (see, for example, \cite[Theorem~B]{KT17}).
	
	We ought to mention that the embedding $X \subset Y$ between two Banach ideal spaces is {\it always} continuous, that is,
	$\norm{\text{id} \colon X \rightarrow Y} = \sup \{ \norm{f}_Y \colon \norm{f}_X = 1 \}$ is finite.
	To duly emphasize this fact, we shall rather write $X \hookrightarrow Y$.
	Moreover, the symbol $X = Y$ indicates that the spaces $X$ and $Y$ are the same as vector spaces and their norms are equivalent.
	Plainly, $X = Y$ if, and only if, $X \hookrightarrow Y$ and $Y \hookrightarrow X$.
	Occasionally, we will write $X \equiv Y$, understanding that $X = Y$, but this time both norms are even equal.

	We refer to the books by Bennett and Sharpley \cite{BS88}, Brudny{\u \i} and Krugljak \cite{BK91}, and Lindenstrauss and Tzafriri \cite{LT79}
	for a comprehensive information about the theory of Banach ideal spaces and, in particular, rearrangement invariant spaces (see also \cite{RGMP16}).
	An inexhaustible source of information about order continuity is Wnuk's monograph \cite{Wn99}.
	To place the mentioned structures within the general framework of abstract Banach lattices and Riesz spaces we recommend taking a look at \cite{MN91}.
	
	\subsection{Calder{\' o}n--Lozanovski{\u \i} construction} \label{SUBSECTION : Calderon construction}
	
	Let us denote by $\mathcal{U}$ the set of all non-negative, concave and positively homogeneous functions
	$\varrho \colon [0,\infty) \times [0,\infty) \rightarrow [0,\infty]$ which vanish only at $(0,0)$. 
	For a function $\varrho$ from $\mathcal{U}$ and two Banach ideal spaces $X$ and $Y$, both defined on the same measure space,
	by the {\bf Calder{\' o}n--Lozanovski{\u \i} construction $\varrho(X,Y)$} (or just the {\bf Calder{\' o}n--Lozanovski{\u \i} space})
	we understand a vector space consisting of all measurable functions, say $f$, such that
	\begin{equation*}
		\abs{f(t)} \leqslant \lambda \varrho(\abs{g(t)},\abs{h(t)})
	\end{equation*}
	for some $\lambda > 0$ with $g \in \text{Ball}(X)$ and $h \in \text{Ball}(Y)$; the norm $\norm{f}_{\varrho(X,Y)}$ of a function $f$
	from $\varrho(X,Y)$ is defined as the infimum over all $\lambda > 0$ for which the above inequality holds.
	It is straightforward to see that
	\begin{equation*}
		\norm{f}_{\varrho(X,Y)} = \inf \, \, \max \left\{ \norm{g}_X, \norm{h}_Y \right\},
	\end{equation*}
	where the infimum is taken over all $g \in X$ and $h \in Y$ with $\abs{f(t)} \leqslant \varrho(\abs{g(t)},\abs{h(t)})$.
	This construction was introduced in the mid 60's by Alberto Calder{\' o}n \cite{Cal64} and later systematically developed by
	Grigorii Lozanovski{\u \i} in a series of papers (see, for example, \cite{Loz71}, \cite{Loz73} and \cite{Loz78}; cf. \cite{KL10} and \cite[Section~15]{Mal04}).
	It is inextricably linked to the interpolation theory, because $\varrho(\cdot,\cdot)$ is an interpolation functor for positive\footnote{We
	can move from positive to arbitrary linear operators by paying the price of assuming the Fatou property.} linear operators.
	Although, we will be not interested in this general construction {\it per se}, its unifying character is hard to overestimate.
	
	More precisely, when $\varrho(s,t) = tF^{-1}(s/t)$ for $s, t > 0$, where $F^{-1}$ is the right-continuous inverse\footnote{Recall
	that for a given Young function $F$ its right-continuous inverse $F^{-1} \colon [0,\infty] \rightarrow [0,\infty]$ is defined as
	follows $F^{-1}(s) \coloneqq \inf \{ t \geqslant 0 \colon F(t) > s\}$ for $0 \leqslant s < \infty$ and $F^{-1}(\infty) \coloneqq \lim_{s \rightarrow \infty} F^{-1}(s)$
	(with the convention that $\inf \{ \emptyset \} = \infty$).}
	of a Young function $F$, the corresponding Calder{\' o}n--Lozanovski{\u \i} space $\varrho(X,L_{\infty})$ is usually denoted by $X_F$
	and is sometimes called the {\bf generalized Orlicz space} (see \cite[Example~2, p.~178]{Mal04}).
	It is only a matter of simple calculations to be convinced that the space $X_F$ consists of those measurable functions $f$ such that
	$F(\abs{f(\cdot)}/\lambda)$ belongs to $X$ for some $\lambda > 0$ (see, for example, \cite[Example~2, p.~178]{Mal04}).
	Moreover,
	\begin{equation*}
		\norm{f}_{\varrho(X,L_{\infty})}
			= \norm{f}_{X_F}
			\coloneqq \inf \left\{ \lambda > 0 \colon \mathcal{M}_F(f/\lambda) \leqslant 1 \right\},
	\end{equation*}
	where $\mathcal{M}_F(f) \coloneqq \norm{F \left( \abs{f} \right)}_X$ is the {\bf modular}
	(pedantically speaking, $\mathcal{M}_F$ should be called a convex and left-continuous semi-modular).
	Here, we follow the convention that if $F(\abs{f}) \notin X$, then $\mathcal{M}_F(f) = \infty$.
	Fortunately, for our purposes, the voluminous theory of modular spaces boils down to the following three simple relations:
	\begin{equation} \label{EQ : norm-modular relation 1}
		\mathcal{M}_F(f) = 1 \quad \Longrightarrow \quad \norm{f}_{X_F} = 1;
	\end{equation}
	\begin{equation} \label{EQ : norm-modular relation 2}
		\norm{f}_{X_F} \leqslant 1 \quad \Longleftrightarrow \quad \mathcal{M}_F(f) \leqslant 1 \quad (\text{see \cite[Theorem~1.4(b), p.~9]{Mal04}});
	\end{equation}
	\begin{equation} \label{EQ : norm-modular relation 3}
		\norm{f}_{X_F} < 1 \quad \Longrightarrow \quad \mathcal{M}_F(f) \leqslant \norm{f}_{X_F} \quad (\text{see the proof of Theorem~3 in \cite{DR00}}).
	\end{equation}
	In general, none of the above implications can be reversed (for example, this can be done in \eqref{EQ : norm-modular relation 1} if,
	and only if, the space $X_F$ satisfies the so-called {\bf norm-modular condition}; see \cite[Remark~26]{KL10} and references given there).
	It is known that the construction $X_F$ inherits many\footnote{However, this is not always that obvious. For example, the space $X_F$ is
	separable provided the space $X$ is separable and $F$ satisfies the appropriately understood $\Delta_2$-condition. Let alone some geometric
	properties like rotundity or uniform convexity (much work in this direction was done between 1990 and 2010; see \cite{KL10} and their references).}
	properties of the space $X$.
	For example, if the space $X$ has the Fatou property or is rearrangement invariant then the same can be said about the space $X_F$.
	Note also that
	\begin{equation} \label{EQ : fundamental function of X_F}
		\psi_{X_F}(t) = \frac{1}{F^{-1}(1/\psi_X(t))}
	\end{equation}
	for $t \in \mathbb{I}$ (cf. \cite[Corollary~4, p.~58]{Mal04}).

	In particular, the space $\varrho(L_1,L_{\infty}) = (L_1)_F$ coincides, up to the equality of norms, with the familiar {\bf Orlicz space $L_F$},
	whilst the space $\varrho(\Lambda(w),L_{\infty}) = (\Lambda(w))_F$ coincide with the {\bf Orlicz--Lorentz space $\Lambda_{F}(w)$}
	(see Section~\ref{SUBSECTION : Orlicz--Lorentz} for more details).
	Evidently, if the Young function $F$ is just a power function, that is, $F(t) = t^p$ for some $1 \leqslant p < \infty$, then the Orlicz space $L_F$
	is nothing else but the classical Lebesgue space $L_p$.
	
	There are two more constructions that should be mentioned.
	In the case of power functions, that is, when $\varrho(s,t) = s^{1-\theta}t^{\theta}$ for some $0 \leqslant \theta \leqslant 1$,
	the space $\varrho(X,Y)$ coincide with the so-called {\bf Calder{\' o}n product} $X^{1-\theta}Y^{\theta}$ (see \cite{Cal64} and \cite[p.~176]{Ma89}).
	In particular, for $1 < p < \infty$, the {\bf $p$-convexification $X^{(p)}$} of $X$, is defined as
	\begin{equation*}
		X^{1/p}(L_{\infty})^{1-1/p} \equiv X^{(p)} \coloneqq \left\{ f \in L_0 \colon \abs{f}^p \in X \right\}
	\end{equation*}
	with $\norm{f}_{X^{(p)}} \coloneqq \norm{\abs{f}^p}^{1/p}_X$ (for more about $p$-convexification of Banach ideal spaces see, for example, \cite{MP89}).
	Let us add that the above construction $X^{(p)}$ makes sense for $0 < p < 1$, but then it is usually referred to as the {\bf $p$-concavification}.
	
	Much more information about Calder{\' o}n--Lozanovski{\u \i} spaces (but also modular spaces and, in particular, Orlicz spaces),
	their Banach space structure and connections with interpolation theory can be found in Maligranda's book \cite{Ma89}
	(see also \cite{BM05}, \cite{Mal04}, \cite{KL10}, \cite{RGMP16} and references therein).
	
	\subsection{Pointwise multipliers} \label{SUBSECTION : Multipliers}
	
	By the space of {\bf pointwise multipliers} $M(X,Y)$ between two Banach ideal spaces $X$ and $Y$ we understand a vector space
	\begin{equation*}
		M(X,Y) \coloneqq \left\{ f \in L_0 \colon fg \in Y \text{ for all } g \in X \right\}
	\end{equation*}
	furnished with the natural\footnote{After all, every function $f \in M(X,Y)$ induces the multiplication operator $M_f \colon X \rightarrow Y$
	given as $M_f \colon g \mapsto fg$ and, moreover, there holds $\norm{f}_{M(X,Y)} = \norm{M_f \colon X \rightarrow Y}$.}
	operator norm $\norm{f}_{M(X,Y)} \coloneqq \sup_{\norm{g}_X = 1} \norm{fg}_Y$.
	Cooked in this fashion, the space $M(X,Y)$ becomes a Banach ideal space itself (see \cite[Proposition~2]{MP89}).
	Note that the space $M(X,Y)$ is non-trivial if, and only if, $X \overset{\textit{\tiny locally}}{\hookrightarrow} Y$, that is,
	for any $f \in X$, and any set $A$ with positive but finite measure, it follows that $\norm{f {\bf 1}_A}_Y \leqslant C \norm{f {\bf 1}_A}_X$
	for some constant $C > 0$ dependent only(!) on $A$ (see \cite[Proposition~2.3]{KLM12}).
	Moreover, the space $M(X,Y)$ is a rearrangement invariant provided $X$ and $Y$ are rearrangement invariant too
	(see \cite[Theorem~2.2]{KLM12} and \cite[Lemma~4.3]{KT22}).
	To put this construction on familiar ground, let us observe that the space $M(X,L_1)$ coincide, up to the equality of norms,
	with the K{\" o}the dual $X^{\times}$ of $X$.
	This leads, among other things, to a general variant of {\bf H{\" o}lder--Rogers inequality}
	\begin{equation} \label{EQ : general Holder--Rogers}
		\norm{fg}_Y \leqslant \norm{f}_X \norm{g}_{M(X,Y)}.
	\end{equation}
	It is also straightforward to see that $M(X,X) \equiv L_{\infty}$ (see \cite[Theorem~1]{MP89}).
	Informally, the space $M(X,Y)$ may be regarded as a \enquote{pointwise quotient} of the space $Y$ by $X$.
	
	Using this construction one can, for example, provide a characterization of compact and weakly compact multiplication
	operators acting between two Banach ideal spaces or describe compact Fourier multipliers acting on Banach spaces of analytic functions (see \cite{KT22} for more).
	
	It is also worth mentioning that a lot has been said about pointwise multipliers acting on function spaces with some smoothness,
	like Besov--Sobolev--Triebel--Lizorkin' type spaces (see, for example, \cite{Sic99} and \cite{Tri03}; classical book on this topic is \cite{MS85}).
	
	There is an extensive literature devoted to this topic. We refer to
	\cite{And60}, \cite{Be96}, \cite{Ber23}, \cite{BL93}, \cite{CDSP08}, \cite{Cro69}, \cite{DSP10}, \cite{DR00}, \cite{KT22}, \cite{KLM12}, \cite{LT17}, \cite{LT21},
	\cite{MN10}, \cite{MP89}, \cite{Nak95}, \cite{Nak16}, \cite{ONe65}, \cite{OT72}, \cite{Ray92} and \cite{Sch10} (see also Nakai's survey \cite{Nak17} and references given there).
	
	\subsection{Pointwise products} \label{SUBSECTION : Products}
	
	For two Banach ideal spaces $X$ and $Y$, the {\bf pointwise product space $X \odot Y$} of $X$ and $Y$ is defined as
	\begin{equation*}
		X \odot Y \coloneqq \left\{ gh \colon g \in X \text{ and } h \in Y \right\},
	\end{equation*}
	and endowed with the \emph{quasi}\footnote{The quasi-norm is, to put it briefly, just the norm in which the $\triangle$-inequality holds but with a constant
	greater than 1 (see, for example, \cite{Kal03} and references therein for much more about quasi-norms and quasi-Banach spaces).}-norm
	\begin{equation}
		\norm{f}_{X \odot Y} \coloneqq \inf \bigl\{ \norm{g}_X \norm{h}_Y \colon f = gh,\, g \in X \text{ and } h \in Y \bigr\}.\label{norma_odot}
	\end{equation}
	The fact that $X \odot Y$ is a vector space is not entirely obvious but follows, in one way or another, from the ideal property of $X$ and $Y$.
	It seems noteworthy that if both spaces $X$ and $Y$ have the Fatou property, then the space $X \odot Y$ has the Fatou property as well.
	Similarly, the space $X \odot Y$ is rearrangement invariant as long as both spaces $X$ and $Y$ are rearrangement invariant.
	All this is essentially due to the fact that the product space $X \odot Y$ is isometric to the $\frac{1}{2}$-concavification
	of the Calder{\' o}n product $X^{1/2}Y^{1/2}$.
	Furthermore, for any two given rearrangement invariant spaces, say $X$ and $Y$, the following nice formula holds
	\begin{equation} \label{EQ : fundamental function of product}
		\psi_{X \odot Y}(t) = \psi_X(t) \psi_Y(t)
	\end{equation}
	for $t \in \mathbb{I}$.
	In other words, the fundamental function $\psi_{X \odot Y}$ of the product space $X \odot Y$ is just a product of fundamental function $\psi_X$
	of $X$ and $\psi_Y$ of $Y$ (see \cite[Theorem~2]{KLM12}).
	
	Plainly, taking pointwise product $X \odot Y$ seems somehow opposite to taking \enquote{pointwise quotient} $M(X,Y)$.
	For this reason, the problem whether $X \odot M(X,Y)$ is the same as $Y$ is not without significance.
	A generic example of this kind is {\bf Lozanovski{\u \i}'s factorization theorem}, which says that $X \odot M(X,L_1) = L_1$
	(we will come back to this problem in Section~\ref{SUBSECTION : On factorization} and say more about factorization).
	
	More about pointwise products and factorization can be found in
	\cite{Be96}, \cite{Bun87}, \cite{CS14}, \cite{CS17}, \cite{Gil81}, \cite{JR76}, \cite{KT24}, \cite{KLM14}, \cite{KLM19}, \cite{LT17}, \cite{LT21},
	\cite{LTJ80}, \cite{Mau74}, \cite{Nil85}, \cite{Rei81} and \cite{Sch10}.
	
	\section{{\bf Main results}} \label{SECTION : Main results}
	
	\subsection{Generalized Young conjugate} \label{SUBSECTION : Young conjugate}
	
	This section essentially can be seen as a spin-off from \cite{LT17}
	(see also \cite{And60}, \cite{Mau74}, \cite[Section~7]{KLM12}, \cite[pp.~77--78]{Mal04}, \cite[pp.~334--335]{MP89} and \cite{ONe65}).
	
	In what follows it will be convenient to write that a Young function $F$ {\bf jumps to infinity} if
	$b_F \coloneqq \sup \{ t \geqslant 0 \colon M(t) < \infty \}$ is finite.
	Otherwise, that is, when $b_M = \infty$, we will say that $M$ is {\bf finite}.
	
	\begin{definition}[Generalized Young conjugate] \label{DEF : generalized Young conjugate}
		{\it For a given two Young functions, say $F$ and $G$, we define the {\bf generalized Young conjugate $G \ominus F$}
		of $G$ with respect to $F$, in the following way
		\begin{equation} \label{EQ : Generalized Young conjugate}
			(G \ominus F)(t) \coloneq \sup\limits_{0 \leqslant s < b_F} \left\{ G(st) - F(s) \right\}
		\end{equation}
		for $t \geqslant 0$.
		Moreover, for $0 < a < b_F$, we define the function $G \ominus_a F$, which can be seen as a truncated version of $G \ominus F$,
		as follows
		\begin{equation} \label{EQ : truncated GYCF}
			(G \ominus_a F)(t) \coloneq \sup_{0 \leqslant s \leqslant a} \left\{ G(st) - F(s) \right\}
		\end{equation}
		for $t \geqslant 0$.}
	\end{definition}

	One can show that both functions $G \ominus F$ and $G \ominus_a F$ are again Young functions (see, for example, \cite[Theorem~4]{And60} and \cite[Lemma~2]{LT17}).
	Moreover, for $t \geqslant 0$, there holds
	\begin{equation} \label{EQ : granica dopelniajacych}
		\lim_{a \rightarrow b_F^{-}} (G \ominus_a F)(t) = (G \ominus F)(t).
	\end{equation}
	It is also clear that the so-called {\bf generalized Young inequality} holds, that is,
	for any two Young functions $F$ and $G$,
	\begin{equation} \label{EQ : generalized Young inequality}
		G(st) \leqslant (G \ominus_a F)(t) + F(s),
	\end{equation}
	where $t \geqslant 0$ and $0 \leqslant s \leqslant a$.
	
	The idea behind \eqref{EQ : Generalized Young conjugate} and \eqref{EQ : truncated GYCF} goes back to the work of Ando \cite{And60} and Maurey \cite{Mau74}.
	Roughly speaking, both constructions \eqref{EQ : Generalized Young conjugate} and \eqref{EQ : truncated GYCF} are intended to properly generalize
	the K{\" o}the duality theory of Orlicz spaces.
	To see this, recall that the K{\" o}the dual $(L_M)^{\times}$ of the Orlicz space $L_F$ coincides, up to the equivalence of norms,
	with another Orlicz space $L_{F^{*}}$, where the function $F^{*}$ is defined as $F^{*}(t) \coloneqq \sup_{s > 0} \left\{ st - F(s) \right\}$
	for $t \geqslant 0$, and is customarily called the {\bf Young conjugate} of $M$ (all of this is classic; see \cite[p.~175]{And60} and
	\cite[Definition~1.5]{ONe65}; cf. \cite[Chapters~8 and 9]{Ma89}, \cite[Chapter~15]{RGMP16} and their references).
	In other words, we have
	\begin{equation*}
		(L_F)^{\times} = M(L_F,L_1) = L_{F^*} = L_{\text{id} \ominus F}.
	\end{equation*}
	Thus, in general, it is perfectly natural to suspect that
	\begin{equation} \label{EQ : Faktoryzacja Orliczy LT17}
		M(L_F,L_G) = L_{G \ominus F}.
	\end{equation}
	In fact, after many partial results, this conjecture was finally confirmed in full generality by Karol Le{\' s}nik and the second-named author in \cite{LT17}
	(see also \cite{LT21} for further generalization to the setting of Musielak--Orlicz spaces).
	
	Note also that \eqref{EQ : Faktoryzacja Orliczy LT17} in its most rudimentary form, that is, when both Young functions $M$ and $N$ are just a power functions,
	looks as follows: For $1 \leqslant q < p < \infty$, we have $M(L_p,L_q) = L_r$ with $1/r = 1/q - 1/p$.
	
	\subsection{Pointwise multipliers of Calder{\' o}n--Lozanovski{\u \i} spaces} \label{SUBSECTION : PMCLS}
	
	Before we go any further, let us introduce some notation that will make our live a little more bearable.

	\begin{notation}[Nice triple] \label{NOTATION : Nice triple}
		{\it Let $X$ be a rearrangement invariant function space.
		Further, let $F$ and $G$ be two Young functions.
		We will say that the triple $(X,F,G)$ is {\bf nice} if the following three conditions:}
		\begin{itemize}
			\item[($1$)] {\it the fundamental function $\psi_X$ of the space $X$ does not vanish at zero;}
			\item[($2$)] {\it the function $F$ is finite;}
			\item[($3$)] {\it the function $G$ jumps to infinity,}
		\end{itemize}
		{\it do not meet simultaneously.}
	\end{notation}
	
	\begin{remark} \label{REMARK : not nice on [0,1]}
		Unwinding the above definition just a little bit, it is straightforward to see that the triple $(X,F,G)$,
		where $X$ is a rearrangement invariant space $X$ defined on $\mathbb{I} = [0,1]$, is not nice then $X = X_F = X_G = L_{\infty}[0,1]$
		regardless of the functions $F$ and $G$. 
	\end{remark}

	After this modest preparation, we are finally ready to prove the following
	
	\begin{theorem}[Pointwise multipliers of Calder{\' o}n--Lozanovski{\u \i} spaces] \label{THEOREM : PM between CL}
		{\it Let $X$ be a rearrangement invariant function space.
		Further, let $F$ and $G$ be two Young functions.
		Then $M(X_F, X_G) = X_{G \ominus F}$ if, and only if, the triple $(X,F,G)$ is \enquote{nice} (see Notation~\ref{NOTATION : Nice triple} for clarification).}
	\end{theorem}
	\begin{proof}
		{\bf Proof of necessity.} Suppose that the triple $(X,F,G)$ fails to be nice.
		We have to prove that
		\begin{equation} \label{EQ : necessity} \tag{$\clubsuit$}
			M(X_F, X_G) \neq X_{G \ominus F}.
		\end{equation}
		Actually, a simple plan to justify \eqref{EQ : necessity} is to show that the space $X_{G \ominus F}$ is trivial and $M(X_F, X_G)$ is not.
		The former is easy.
		Indeed, since the function $G$ jump to infinity and $F$ is finite, so the Young conjugate function $G \ominus F$ is identically equal
		to infinity outside zero.
		This means that the space $X_{G \ominus F}$ can only contain one function, namely, the one that is equal to zero almost everywhere.
		With this in mind, let us make one more observation.
		Since the fundamental function $\psi_X$ of the space $X$ does not vanish at zero, so it follows that $X \hookrightarrow L_{\infty}$.
		Thus, also $X_F \hookrightarrow (L_{\infty})_F = L_{\infty}$.
		We have
		\begin{align*}
			\norm{{\bf 1}_{[0,1]}}_{M(X_F, X_G)}
				& = \sup \{ \norm{f {\bf 1}_{[0,1]}}_{X_G} \colon f \in \text{Ball}(X_F) \} \\
				& \leqslant \norm{X_F \hookrightarrow L_{\infty}} \, \sup \{ \norm{f {\bf 1}_{[0,1]}}_{X_G} \colon f \in \text{Ball}(L_{\infty}) \} \\
				& \leqslant \norm{X_F \hookrightarrow L_{\infty}} \, \norm{{\bf 1}_{[0,1]}}_{X_G} \\
				& = \norm{X_F \hookrightarrow L_{\infty}} \left[ G^{-1}\left( 1 / \psi_X(1) \right) \right]^{-1},
		\end{align*}
		where the last equality is due to \eqref{EQ : fundamental function of X_F}.
		This means, of course, that the space $M(X_F, X_G)$ is non-trivial and $\eqref{EQ : necessity}$ follows.
		
		{\bf Proof of sufficiency.} Suppose that the triple $(X,F,G)$ is nice.
		This is, unfortunately, the moment where things start to get a little more cumbersome, because formally we have five different situations to consider.
		However, be not of faint heart, we can divide the whole argument into just two cases, one of which is almost obvious.
		
		$\bigstar$ The case when $b_F = \infty$, $b_G < \infty$ and $\psi_X(0) = 0$.
		From what we said above, we already know that in this situation the space $X_{G \ominus F}$ is trivial.
		Moreover, due to our assumptions, $X_F \not\hookrightarrow L_{\infty}$ and $X_G \hookrightarrow L_{\infty}$.
		But this means that the space $M(X_F,X_G)$ is trivial as well.
		(Otherwise, we would have that $\norm{f {\bf 1}_{[0,1]}}_{L_{\infty}} \leqslant C \norm{f {\bf 1}_{[0,1]}}_{X_F}$, which is clearly impossible,
		because the space $X_F$ contains unbounded functions.)
		There is no doubt that $0 = 0$.
		
		$\bigstar$ The remaining case.
		We can assume right away that $(G \ominus F)(t)$ is finite for some $t > 0$, because the opposite is only possible when
		$b_G < \infty$ and $b_F = \infty$.
		Therefore, the space $X_{G \ominus F}$ is non-trivial.
		Moreover, without any loss of generality we can assume that $b_F \geqslant 1$.
		Let us start with the embedding
		\begin{equation} \tag{$\heartsuit$} \label{EQ : wlozenie 1}
			X_{G \ominus F} \hookrightarrow M(X_F,X_G).
		\end{equation}
		Take $f \in X_{G \ominus F}$ with $\norm{f}_{X_{G \ominus F}} \leqslant 1/2$ and $g \in X_F$ with $\norm{g}_{X_F} \leqslant 1/2$.
		We have
		\begin{align*}
			\mathcal{M}_G(fg)
			& = \norm{G(\abs{fg})}_X \\
			& \leqslant\norm{(G \ominus F)(\abs{f}) + F(\abs{g})}_X \quad \quad (\text{using \eqref{EQ : generalized Young inequality}}) \\
			& \leqslant\norm{(G \ominus F)(\abs{f})}_X + \norm{F(\abs{g})}_X \quad \quad (\text{by $\triangle$-inequality}) \\
			& \leqslant 1.
		\end{align*}
		Consequently, using \eqref{EQ : norm-modular relation 2}, $fg \in X_{G}$ with $\norm{fg}_{X_G} \leqslant 1$ and \eqref{EQ : wlozenie 1} follows.
		Now, let us focus on the opposite embedding
		\begin{equation} \tag{$\diamondsuit$} \label{EQ : wlozenie 2}
			M(X_F,X_G) \hookrightarrow X_{G \ominus F}.
		\end{equation}
		Note, that in order to show \eqref{EQ : wlozenie 2}, it is enough to prove something a little easier,
		namely that there is a constant $C > 0$ such that for all positive simple functions, say $f$, the following
		inequality
		\begin{equation} \label{EQ : redukacja 1}
			\norm{f}_{X_{G \ominus F}} \leqslant C \norm{f}_{M(X_F,X_G)}
		\end{equation}
		holds.
		Indeed, a straightforward argument based on the Fatou property of the space $X$ will do the job.
		Observe, however, that we can make one more reduction and instead of proving \eqref{EQ : redukacja 1}, we can
		only show that for any $1 < a < b_F$ the following modular inequality
		\begin{equation} \label{EQ : to chcemy pokazac}
			\mathcal{M}_{G \ominus_a F}(f) \leqslant \frac{1}{2}
		\end{equation}
		holds.
		To see this, let $f$ be a positive simple function, that is, $f = \sum_{n=1}^N a_n {\bf 1}_{A_n}$,
		where $a_n$'s are positive reals and $A_n$'s are disjoint sets of positive but finite measure.
		Without the loss of generality we can assume that
		\begin{equation} \label{EQ : redukacja 2}
			\norm{f}_{M(X_F,X_G)} \leqslant \frac{1}{2 \delta},
		\end{equation}
		where $\delta = b_F$ provided $F$ jumps to infinity or $\delta = 1$ otherwise.
		Remembering about \eqref{EQ : granica dopelniajacych} and using \eqref{EQ : to chcemy pokazac}, we have
		\begin{equation*}
			\mathcal{M}_{G \ominus F}(f)
				= \norm{(G \ominus F)(f)}_X
				= \liminf_{a \rightarrow b_{F}^{-}} \norm{(G \ominus_a F)(f)}_X
				\leqslant \frac{1}{2}.
		\end{equation*}
		Thus, due to \eqref{EQ : norm-modular relation 2}, we have $\norm{f}_{X_{G \ominus F}} \leqslant 1$ and \eqref{EQ : redukacja 1}
		follows with $C = 2 \delta$.
		To recap, form this point on, we can focus all our efforts on showing \eqref{EQ : to chcemy pokazac}.
		Fix $1 < a < b_F$.
		It is easy to see that for any $1 \leqslant n \leqslant N$ there is  $0<b_n\leqslant a$ with
		\begin{equation} \label{EQ : przed}
			G(a_n b_n) = (G \ominus_a F)(a_n) + F(b_n).
		\end{equation}
		Set
		\begin{equation} \label{EQ : definition of g}
			g \coloneqq \sum_{n=1}^{N} b_n {\bf 1}_{A_n}.
		\end{equation}
		Then, in view of \eqref{EQ : przed},
		\begin{equation} \label{EQ : po}
			G(fg) = (G \ominus_a F)(f) + F(g).
		\end{equation}
		We claim that
		\begin{equation} \label{EQ : juz ostatnie}
			\norm{g}_{X_F} \leqslant 1.
		\end{equation}
		We will justify this in three steps.
		
		{\bf Step 1: Divide and conquer.}
		We are going to divide the support $\mathbb{I}$ of the space $X$ into a sequence $\{ \Delta_n \}_{n=1}^{\infty}$
		of pairwise disjoint sets with finite measure.
		To be precise here, in the case when $\mathbb{I} = [0,\infty)$ we actually need an infinite number of pieces of positive measure,
		whereas in the case when $\mathbb{I} = [0,1]$ we will be satisfied with only a finite number of $\Delta_n$'s having positive measure.
		Further, the division algorithm will depend on two factors, namely, whether $F$ is finite or not and whether $\psi_X(0^{+})$
		is equal to or greater than zero.
		We will explain how to do this in the simplest situation, that is, when both $F$ and $G$ are finite and $\psi_X(0^{+}) = 0$.
		Later we will show how to modify this construction to work in other cases as well.
		Let us start putting our plans into action.
		Since $\psi_{X_F}(0^{+}) = 0$, so simple functions are order continuous in $X$.
		In consequence, there is $\lambda > 0$ with $\norm{{\bf 1}_{\Delta}}_{X_F} \leqslant 1/a$ for all $\Delta \subset \mathbb{I}$
		with $m(\Delta) \leqslant \lambda$.
		Plainly, since the space $X$ is rearrangement invariant, so $\mathbb{I} = \bigcup_{n=1}^{\infty} \Delta_n$,
		where $\Delta_n \coloneqq [(n-1)\lambda,n\lambda] \cap \mathbb{I}$ for $n \in \mathbb{N}$.
		Moreover, $\norm{{\bf 1}_{\Delta_n}}_{X_F} \leqslant 1/a$ for all $n \in \mathbb{N}$.
		Now, suppose that $b_F < \infty$ and $\psi_X(0^{+}) \geqslant 0$.
		Without any loss of the generality, we can assume that $b_F > 1$ and $\psi_X(0^{+}) \leqslant 1$.
		Then $\psi_{X_F}(0^{+}) \leqslant 1/b_F < 1$.
		Thus, again, we can find $\lambda > 0$ with $\norm{{\bf 1}_{\Delta}}_{X_F} \leqslant 1$ for all $\Delta \subset \mathbb{I}$ with $m(\Delta) \leqslant \lambda$.
		The rest is already known.
		Finally, let us discuss the case when $\psi_{X}(0^{+}) > 0$ but both $F$ and $G$ are finite.
		As before, we can assume that $\psi_{X_F}(0^{+}) = 1 / F^{-1}(1/\psi_X(0^+)) < 1$.
		In consequence, there is $\lambda > 0$ with $\norm{{\bf 1}_{\Delta}}_{X_F} \leqslant 1$ whenever $\Delta \subset \mathbb{I}$ with $m(\Delta) \leqslant \lambda$.
		Again, the rest is clear.
		The algorithm is finished.
		
		Next, it follows from the very definition of $g$ that $\abs{g(t)} \leqslant a < b_F$.
		Indeed, according to \eqref{EQ : definition of g}, $g$ is just a simple function of the form $g = \sum_{n=1}^N b_n {\bf 1}_{A_n}$
		with all $b_n$'s not exceeding $a < b_F$ (for this, take a look at \eqref{EQ : przed} and what precedes it).
		So, without forgetting about \eqref{EQ : redukacja 2}, we get
		\begin{equation} \label{EQ : ABDASFMKA}
			\norm{fg {\bf 1}_{\Delta_n}}_{X_G}
				\leqslant \norm{f}_{M(X_F,X_G)} \norm{g {\bf 1}_{\Delta_n}}_{X_F}
				\leqslant \frac{a}{2 \delta} \norm{{\bf 1}_{\Delta_n}}_{X_F}
				\leqslant \frac{1}{2}.
		\end{equation}
		Using \eqref{EQ : norm-modular relation 3} along with \eqref{EQ : po}, we have
		\begin{equation} \label{EQ : D&C}
			\mathcal{M}_F(g {\bf 1}_{\Delta_n})
				\leqslant \mathcal{M}_G(fg {\bf 1}_{\Delta_n})
				\leqslant \norm{fg {\bf 1}_{\Delta_n}}_{X_G}
				\leqslant \frac{1}{2},
		\end{equation}
		where the last inequality follows from \eqref{EQ : ABDASFMKA}.
		Finally, for $n \in \mathbb{N}$, set
		\begin{equation*}
			g_n \coloneqq \sum_{i=1}^n g {\bf 1}_{\Delta_i}.
		\end{equation*}
		{\bf End of Step 1.}
		
		{\bf Step 2: Inductive argument.}
		Now, we will show that
		\begin{equation} \label{EQ : inductvie}
			\mathcal{M}_F(g_n) \leqslant \frac{1}{2}
		\end{equation}
		for all $n \in \mathbb{N}$. For $n = 1$ this follows from \eqref{EQ : D&C}.
		Thus, let $n \geqslant 2$ and assume that $\mathcal{M}_F(g_{n-1}) \leqslant 1/2$.
		Then
		\begin{equation*}
			\mathcal{M}_F(g_n) = \mathcal{M}_F(g_{n-1}) + \mathcal{M}_F(g {\bf 1}_{\Delta_n}) \leqslant 1.
		\end{equation*}
		In consequence, due to \eqref{EQ : norm-modular relation 2},
		\begin{equation} \label{EQ : potrzebne}
			\norm{g_n}_{X_F} \leqslant 1.
		\end{equation}
		We have
		\begin{align*}
			\mathcal{M}_F(g_n)
				& \leqslant \mathcal{M}_G(fg_n) \quad \quad (\text{since, due to \eqref{EQ : po}, $F(g) \leqslant M(fg)$}) \\
				& \leqslant \norm{fg_n}_{X_G} \quad \quad (\text{in view of \eqref{EQ : norm-modular relation 3}, \eqref{EQ : redukacja 2} and \eqref{EQ : potrzebne}}) \\
				& \leqslant \norm{f}_{M(X_F,X_G)} \norm{g_n}_{X_F} \quad \quad (\text{by the H{\" o}lder--Rogers inequality \eqref{EQ : general Holder--Rogers}}) \\
				& \leqslant \frac{1}{2} \quad \quad (\text{by \eqref{EQ : redukacja 2} together with \eqref{EQ : potrzebne}}).
		\end{align*}
		In other words, \eqref{EQ : inductvie} follows. {\bf End of Step 2.}
		
		{\bf Step 3: Limit argument.}
		We know, thanks to \eqref{EQ : inductvie}, that $\mathcal{M}_F(g_n) \leqslant 1/2$ for all $n \in \mathbb{N}$.
		Thus, invoking the Fatou property, we get
		\begin{equation} \label{EQ : uzyte Fatou}
			\norm{g}_{X_F} = \sup \{ \norm{g_n}_{X_F} \colon n \in \mathbb{N} \} \leqslant 1.
		\end{equation}
		But this is exactly what we wanted: \eqref{EQ : juz ostatnie} follows. {\bf End of Step 3.}
				
		Having \eqref{EQ : juz ostatnie} in hand, we can finally finish the whole proof by showing \eqref{EQ : to chcemy pokazac}.
		This is how it goes
		\begin{align*}
			\mathcal{M}_{G \ominus_a F}(f)
				& \leqslant \mathcal{M}_G(fg) \quad \quad (\text{since $(G \ominus_a F)(f) \leqslant G(fg)$}) \\
				& \leqslant \norm{fg}_{X_G} \quad \quad (\text{using \eqref{EQ : norm-modular relation 3}, \eqref{EQ : redukacja 2} and \eqref{EQ : uzyte Fatou}}) \\
				& \leqslant \norm{f}_{M(X_F,X_G)} \norm{g}_{X_F} \quad \quad (\text{by the H{\" o}lder--Rogers inequality \eqref{EQ : general Holder--Rogers}}) \\
				& \leqslant \frac{1}{2} \quad \quad (\text{using \eqref{EQ : redukacja 2} and \eqref{EQ : juz ostatnie}}).
		\end{align*}
		The proof has been completed.
	\end{proof}

	\begin{remark}[What if the triple $(X,F,G)$ is not nice?] \label{REMARK : not nice}
		Note that it follows from the proof of Theorem~\ref{THEOREM : PM between CL} that if the triple $(X,F,G)$ fails to be nice,
		then although $M(X_F,X_G) \neq X_{G \ominus F}$, the space $M(X_F,X_G)$ is nevertheless non-trivial.
		The question that certainly looms on the horizon is: {\it How to describe the space $M(X_F,X_G)$?}
		In the case when $X$ is defined on the unit interval $\mathbb{I} = [0,1]$, things are rather straightforward.
		In fact, since the fundamental function $\psi_X$ does not vanish at zero, so $X = L_{\infty}[0,1]$ and
		\begin{align*}
			M(X_F,X_G)
				& \equiv M((L_{\infty}[0,1])_F, (L_{\infty}[0,1])_G) \\
				& = M(L_{\infty}[0,1], L_{\infty}[0,1]) \\
				& \equiv L_{\infty}[0,1].
		\end{align*}
		However, the remaining case, when $X$ is defined on the half-line $\mathbb{I} = [0,\infty)$, is definitely less obvious.
	\end{remark}

	The next result clarifies this situation.

	\begin{theorem} \label{THEOREM : main thm not nice}
		{\it Let $X$ be a rearrangement invariant space.
		Further, let $F$ and $G$ be two Young functions.
		Suppose that either the triple $(X,F,G)$ is not \enquote{nice} (see Notation~\ref{NOTATION : Nice triple} for clarification)
		or the space $X$ is a sequence space.
		Then $M(X_{F},X_{G}) = X_{G \ominus_1 F}$.}
	\end{theorem}
	\begin{proof}
		To facilitate the impending maneuvers, without any loss of generality, we may assume that
		\begin{enumerate}
			\item[$\bullet$] $F(1) = G(1) = 1$;
			\item[$\bullet$] $b_G > 1$ and $\psi_X(0) = 1$ provided $X$ is a function space;
			\item[$\bullet$] $\psi_X(1) = 1$ provided $X$ is a sequence space.
		\end{enumerate}
		Moreover, by Remark~\ref{REMARK : not nice}, we can ignore the case when the space $X$ is defined on $\mathbb{I} = [0,1]$.
		Now, we will consider both embeddings separately.
	
		{\bf Embedding $X_{G \ominus_1 F} \hookrightarrow M(X_{F},X_{G})$.}
		Take $f \in X_{G \ominus_1 F}$ with $\norm{f}_{X_{G \ominus_1 F}} \leqslant 1/2$ and $g \in X_{F}$ with $\norm{g}_{X_F} \leqslant 1/2$. 
		Our plan is to show that the product $f g$ belongs to $X_{G}$ with $\norm{f g}_{X_G} \leqslant 1$. 
		Recall the following variant of Young's inequality
		\begin{equation} \label{EQ : GYI01}
			G(st) \leqslant (G \ominus_1 F)(s) + F(t),
		\end{equation}
		where $0 \leqslant s,t \leqslant 1$, follows directly from the definition of the function $G \ominus_1 F$
		(see Definition~\ref{DEF : generalized Young conjugate} and \eqref{EQ : generalized Young inequality}).
		Moreover, in view of our assumptions, $\norm{X \hookrightarrow L_{\infty}} = 1$, so $\norm{f}_{L_{\infty}} \leqslant 1$
		and $\norm{g}_{L_{\infty}} \leqslant 1$.
		We have
		\begin{align*}
			\mathcal{M}_G(fg)
				& = \norm{G(\abs{fg})}_X \\
				& \leqslant\norm{(G \ominus_1 F)(\abs{f}) + F(\abs{g})}_X \quad \quad (\text{using \eqref{EQ : GYI01}}) \\
				& \leqslant\norm{(G \ominus_1 F)(\abs{f})}_X + \norm{F(\abs{g})}_X \quad \quad (\text{via $\triangle$-inequality}) \\
				& \leqslant\frac{1}{2} + \frac{1}{2}.
		\end{align*}
		Here, the last inequality is due to the well-known relation between the norm and the modular \eqref{EQ : norm-modular relation 3}.
		This means that indeed $fg \in X_{G}$ with $\norm{fg}_{X_G} \leqslant 1$.
		That is all we wanted for now.
		
		{\bf Embedding $M(X_{F},X_{G}) \hookrightarrow X_{G \ominus_1 F}$.}
		Note that we only need to show that the following inequality
		\begin{equation} \label{orliczfp_notnice} \tag{$\spadesuit$}
			\norm{f}_{X_{G \ominus_1 F}} \leqslant C \norm{f}_{M(X_{F},X_{G})}
		\end{equation}
		holds for some constant $C > 0$ and all positive, simple functions $f$.
		Then, using the Fatou property, we can lift this inequality to the whole space $M(X_{F},X_{G})$.
		Keeping this in mind, let $f$ be a positive, simple function with $\norm{f}_{M(X_{F},X_{G})} \leqslant \frac{1}{2}$.
		Plainly, there is a sequence $\{ a_n \}_{n=1}^N$ of positive reals together with a sequence $\{ A_n \}_{n=1}^{\infty}$ of sets with
		positive but finite measure, such that $f = \sum_{n=1}^N a_n {\bf 1}_{A_n}$.
		Since $M(X_F,L_{\infty}) \equiv L_{\infty}$ and, due to our assumptions, $\norm{M(X_F,X_G) \hookrightarrow M(X_F,L_{\infty})} \leqslant 1$,
		so $\norm{f}_{L_{\infty}} \leqslant 1$.
		Moreover, using the standard compactness argument, for every $n \in \mathbb{N}$, we can find  $0 \leqslant b_k \leqslant 1$ with
		\begin{equation*}
			G(a_n b_n) = (G \ominus_1 F)(a_n) + F(b_n).
		\end{equation*}
		Set $g \coloneqq \sum_{n=1}^{\infty} b_n {\bf 1}_{A_n}$. We claim that
		\begin{equation} \label{norm of g_notnice}
			\norm{g}_{X_F} \leqslant 1.
		\end{equation}
		To see this, we will divide the measure space underlying $X$ into a sequence $\{ \Delta_n \}_{n=1}^{\infty}$
		of pairwise disjoint sets with positive and finite measure.
		Let us explain how to do this.
		
		$\bigstar$ The situation in which $X$ is a sequence space is the most straightforward.
		Just take $\Delta_n$ as the $n^{\text{th}}$ atom, that is, $\Delta_n = \{ n \}$ for $n \in \mathbb{N}$.
		Then, making use of \eqref{EQ : fundamental function of X_F} along with the fact that the space $X$ is rearrangement invariant,
		we infer that
		\begin{equation*}
			\norm{{\bf 1}_{\{ n \}}}_{X_F}
				= \norm{{\bf 1}_{\{ 1 \}}}_{X_F}
				= [F^{-1}(1/\psi_X(1))]^{-1}
				= [F^{-1}(1)]^{-1}
				\leqslant 1.
		\end{equation*}
		
		$\bigstar$ Next, suppose that $X$ is a function space. It is enough to take $\Delta_n = [n-1,n]$ for $n \in \mathbb{N}$.
		Then, exactly as above,
		\begin{equation*}
			\norm{{\bf 1}_{[n-1,n]}}_{X_F}
				= \norm{{\bf 1}_{[0,1]}}_{X_F}
				= [F^{-1}(1/\psi_X(1))]^{-1}
				\leqslant [F^{-1}(1)]^{-1}
				\leqslant 1.
		\end{equation*}
		
		\noindent
		All this is, of course, very similar to the proof of Theorem~\ref{THEOREM : PM between CL} (see \enquote{Step~1}).
		For this reason, to obtain \eqref{norm of g_notnice}, it is enough to mimic the inductive argument in \enquote{Step~2} contained therein.
		However, since repeating all this here may seem rather boring, we leave easy-to-fill details for the inquisitive reader.
		
		Going ahead, we have
		\begin{align*}
			\mathcal{M}_{G \ominus_1 F}(f)
				& \leqslant \mathcal{M}_{G}(fg) \quad \quad (\text{since $(G \ominus_1 F)(f) \leqslant G(fg)$}) \\
				& \leqslant \norm{fg}_{X_G} \quad \quad (\text{using \eqref{EQ : norm-modular relation 3}}) \\
				& \leqslant \norm{f}_{M(X_F,X_G)} \norm{g}_{X_F} \quad \quad (\text{by the H{\" o}lder--Rogers inequality \eqref{EQ : general Holder--Rogers}}) \\
				& \leqslant \frac{1}{2} \quad \quad (\text{in view of \eqref{norm of g_notnice}}).
		\end{align*}
		But this means that \eqref{orliczfp_notnice} holds.
		In consequence, the proof has been completed.
	\end{proof}

	It seems quite instructive to support the above results with some concrete calculations.
	This is by no means difficult, but rather tedious.
	Recall that the {\bf $p^{\text{th}}$-power function $F_p \colon [0,\infty) \rightarrow [0,\infty)$}
	with $1 \leqslant p < \infty$ is defined in the following way
	\begin{equation*}
		F_p(t)\coloneq \frac{1}{p} t^p.
	\end{equation*}

	\begin{example} \label{EXAMPLE : rachuneczki 1}
		Let $1 \leqslant p,q < \infty$ with $1 \leqslant p < q < \infty$ and $1/r = 1/p - 1/q$.
		Our goal is to compute
		\begin{equation*}
			(F_p \ominus F_q)(t) \coloneqq \sup_{s > 0} \left\{ \frac{1}{p}(st)^{p} - \frac{1}{q}s^q \right\}.
		\end{equation*}
		Some elementary calculations show us that for fixed $t > 0$ the extreme value of the function
		\begin{equation*}
			f(s,t) = \frac{1}{p}(st)^{p} - \frac{1}{q}s^q
		\end{equation*}
		is attained at the point $s_{\text{ext}} \coloneqq t^{p / (q-p)}$.
		In consequence, for $t \geqslant 0$, we have
		\begin{align*}
			(F_p\ominus F_q)(t)
				& = \frac{1}{p} (s_{\text{ext}} t)^p - \frac{1}{q} s_{\text{ext}}^q \\
				& = \frac{1}{p} (t^{p / (q-p)} t)^p - \frac{1}{q} (t^{p / (q-p)})^q \\
				& = \frac{1}{p} (t^{(p+q-p) / (q-p)})^p - \frac{1}{q} (t^{pq / (q-p)}) \\
				& = \frac{1}{p} t^r - \frac{1}{q} t^r \\
				& = \frac{1}{r} t^r \\
				& = F_r(t).
		\end{align*}
	\end{example}
	
	\begin{example} \label{EXAMPLE : rachuneczki 2}
		Let $0 < b < \infty$ and $1 \leqslant p,q < \infty$.
		Now, let us modify the function $F_p$ in the following way
		\begin{equation*}
			F_{p,b}(t) \coloneq
			\begin{cases}
				F_p(t) \quad & \text{ if } \quad 0 \leqslant t \leqslant b\\
				\,\,\, \infty \quad & \text{ if } \quad \quad t > b.
			\end{cases}
		\end{equation*}
		It is easy to see that $(F_{p,b} \ominus F_q)(t) = \infty$ for all $t > 0$.
		Thus, as Theorem~\ref{THEOREM : main thm not nice} teaches us, to obtain a non-trivial Young's function we should instead consider
		the function $F_{p,b} \ominus_1 F_q$.
		
		Let us start with the case $1\leqslant p<q\leqslant\infty$.
		The same calculations as in Example~\ref{EXAMPLE : rachuneczki 1} shows that for $0 \leqslant t \leqslant \min\{1,b\}$ and $1/r = 1/p - 1/q$,
		we have
		\begin{equation*}
			(F_{p,b} \ominus_1 F_q)(t) = F_r(t).
		\end{equation*}
		Furthermore, $(F_{p,b} \ominus_1 F_q)(t) = \infty$ for $t > b$.
		On the other hand, for $1 < t \leqslant b$, we have
		\begin{equation*}
			(F_{p,b} \ominus_1 F_q)(t) = \sup_{0 \leqslant s \leqslant 1} \left\{ \frac{1}{p}(st)^p - \frac{1}{q} s^q \right\}.
		\end{equation*}
		Again, due to Example~\ref{EXAMPLE : rachuneczki 1}, we already know that the extreme value of the function $f(s,t) = (st)^p/p - s^q/q$
		with fixed $t > 0$ is attained in the point $s_{\text{ext}} = t^{p / (q-p)} > 1$, that is, outside the range of the parameter $s$.
		Since the function $s \mapsto f(s,t)$ is increasing, so the supremum is attained at the end of the interval $[0,1]$.
		Consequently,
		$$
		(F_{p,b} \ominus_1 F_q)(t) = \frac{1}{p} t^p - \frac{1}{q}.
		$$
		In summary,
		$$
		(F_{p,b} \ominus_1 F_q)(t) =
		\begin{cases}
			\,\,\, F_r(t) \quad & \text{ for } \quad 0 \leqslant t \leqslant \min\{1,b\} \\
			\frac{1}{p} t^p - \frac{1}{q} \quad & \text{ for } \quad \quad 1 < t \leqslant b \\
			\quad \infty \quad & \text{ for } \quad \quad \quad t > b.
		\end{cases}
		$$
		
		Finally, let us consider one more situation when $1 \leqslant q \leqslant p < \infty$.
		In this case $(st)^p / p \leqslant s^q / q$ for $t > 0$ and $0 \leqslant s \leqslant 1$, so
		\begin{equation*}
			(F_{p,b} \ominus_1 F_q)(t) =
			\begin{cases}
				\, 0 \quad & \text{ for } \quad 0 \leqslant t \leqslant b \\
				\infty \quad & \text{ for } \quad \quad t > b.
			\end{cases}
		\end{equation*}
	\end{example}
	
	One of the immediate conclusions from the above result is the following example, first noted by Maligranda and Persson in \cite[Corollary~2]{MP89}.

	\begin{example}[Maligranda and Persson, 1989] \label{EXAMPLE : Maligranda i Persson}
		Let $X$ be a rearrangement invariant space.
		Further, let $1 \leqslant q < p < \infty$ with $1/r = 1/q - 1/p$.
		We claim that
		\begin{equation} \label{EQ : Maligranda i Persson}
			M(X^{(p)}, X^{(q)}) \equiv X^{(r)}.
		\end{equation}
		To see this, note that if $F_p(t) = t^p / p$ and $F_q(t) = t^q / q$, then the Calder{\' o}n--Lozanovski{\u \i}
		constructions $X_{F_p}$ and $X_{F_q}$ coincide with the $p$-convexification $X^{(p)}$ of $X$ and the $q$-convexification $X^{(q)}$ of $X$, respectively.
		Moreover, it is crystal clear that the triple $(X,F_p,F_q)$ is nice.
		Thus, after realizing that $(F_q \ominus F_p)(t) = t^r / r$ (just in case, Example~\ref{EXAMPLE : rachuneczki 1} may be helpful),
		it is enough to call Theorem~\ref{THEOREM : PM between CL} on stage.
		In particular, if $X = L_1$, the formula \eqref{EQ : Maligranda i Persson} reduces to the well-known fact that $M(L_p,L_q) \equiv L_r$.
	\end{example}

	\begin{example} \label{EXAMPLE : MP jak skacze}
		Let $X$ be a rearrangement invariant space with $X \hookrightarrow L_{\infty}$.
		Further, let $1 \leqslant q < p < \infty$ with $1/r = 1/q - 1/p$ and $b = 1$.
		It follows from Theorem~\ref{THEOREM : main thm not nice} and Example~\ref{EXAMPLE : rachuneczki 2} that
		\begin{equation*}
			M(X_{F_{p,b}}, X_{F_{q}}) = X_{F_{r,b}}.
		\end{equation*}
		In particular, for $X = L_1 \cap L_{\infty}$, we get
		\begin{equation*}
			M(L_p \cap L_{\infty}, L_q \cap L_{\infty}) = L_r \cap L_{\infty}.
		\end{equation*}
	\end{example}
	
	\section{{\bf Applications}} \label{SECTION : Applications}
	
	In this section we will show how, from the description of the space of pointwise multipliers between Calder{\' o}n--Lozanovski{\u \i}
	spaces (see Theorem~\ref{THEOREM : PM between CL}), one can deduce the factorization of these spaces.
	
	\subsection{A bird's eye view of factorization} \label{SUBSECTION : Factorization}
	
	Recall that classical {\bf Lozanovski{\u \i}'s factorization} \cite[Theorem~6]{Loz69} (see also \cite{Gil81}, \cite{JR76}, \cite[Example~6, p.~185]{Mal04} and \cite{Rei81})
	teaches us that for any $\varepsilon > 0$ each function $f$ from $L_1$ can be written as a pointwise product of two functions, say $g$ and $h$, one from $X$ and the
	other from the K{\" o}the dual $X^{\times}$ of $X$, in such a way that
	\begin{equation*}
		\norm{f}_{L_1} \leqslant \norm{g}_X \norm{h}_{X^{\times}} \leqslant (1+\varepsilon) \norm{f}_{L_1}.
	\end{equation*}
	Furthermore, knowing that the space $X$ posses the Fatou property, we can set $\varepsilon = 0$.
	In other words, {\bf $L_1$ can be factorized through $X$}, that is,
	\begin{equation*}
		X \odot X^{\times} = X \odot M(X,L_1) = L_1.
	\end{equation*}

	Clearly, one can replace $L_1$ from Lozanovski{\u \i}'s factorization by an arbitrary Banach ideal space $Y$ and ask whether
	{\bf $Y$ can be factorized through $X$}?
	It turns out that an answer to this question is in general very difficult and the equality $X \odot M(X,Y) = Y$ does not hold
	without some extra assumptions on $X$ and $Y$ (see, for example, \cite[Example~2]{KLM14} and \cite[Example~A.2]{KT22}).
	
	The topic of factorization of Banach ideal spaces seems to be very much in vouge recently;
	see, for example, the papers \cite{DR00} and \cite{LT17} for Orlicz spaces; \cite{CS17}, \cite{KLM14} and \cite{Rei81} for Lorentz and Marcinkiewicz spaces;
	\cite{KLM12} for Calder{\' o}n--Lozanovski{\u \i} spaces; \cite{KT22}, \cite{KLM19} and \cite{Sch10} for Ces{\' a}ro and Tandori spaces;
	and \cite{LT21} for Musielak--Orlicz spaces.
	
	Nevertheless, factorization has much more to offer that just a simple analogy with the Lozanovski{\u \i} theorem.
	For example, using some factorization techniques, Nilsson was able to give another proof of Pisier's result (see \cite[Theorem~2.4]{Nil85}).
	Moreover, Odell and Schlumprecht's proof that $\ell_2$ is arbitrarily distortable makes a use of Lozanovski{\u \i}'s factorization
	(see \cite[p.~261]{OS94}).
	There is also a beautiful connection between complex interpolation, Lozanovski{\u \i}'s factorization and the construction of twisted sums of
	Banach spaces (see \cite{CS14}, \cite{Cor22}, \cite{CS17}, \cite{CFG17}, \cite{Kal92} and references therein).
	
	There is, however, a lot of life outside of the \enquote{ideal} world of Banach ideal spaces.
	In the realm of harmonic analysis, for example, (weak) factorization is a domesticated and powerful technique.
	Some remarkable factorization results for Bergman spaces, Hardy spaces and tent spaces, can be found in
	\cite{CRW76}, \cite{CV00}, \cite{Hor77}, \cite{JR76} and \cite{PZ15}.
	
	\subsection{Products of Calder{\' o}n--Lozanovski{\u \i} spaces} \label{SUBSECTION : Products of CL spaces}
	
	The next step, {\it en route} to our factorization results, goes through a series of rather technical lemmas.
	But first, some inevitable notation.
	
	\begin{notation}[Vinogradov's notation] \label{NOTATION : Vinogradov notation}
		{\it For two given quantities, say $A$ and $B$, depending (maybe) on certain parameters, we will write $A \preccurlyeq B$ understanding that
		there exists an absolute constant $C > 0$ (that is, independent of all involved parameters) such that $A \leqslant CB$.
		Moreover, we will write $A \approx B$ meaning that $A \preccurlyeq B$ and $B \preccurlyeq A$.}
	\end{notation}

	\begin{notation}[Small/Large/All arguments]
		{\it Let $F, G \colon [0,\infty) \rightarrow [0,\infty)$ be two real-valued function.
		Further, let the symbol $\square$ denote either the relation $\preccurlyeq$ or $\approx$ (see Notation~\ref{NOTATION : Vinogradov notation}).
		We will say that
		\begin{enumerate}
			\item[$\bullet$] $F \square G$ holds for {\bf small arguments} if there is $T > 0$ such that $F(t) \square G(t)$ for all $0 \leqslant t \leqslant T$;
			\item[$\bullet$] $F \square G$ holds for {\bf large arguments} if there is $T > 0$ such that $F(t) \square G(t)$ for all $t \geqslant T$;
			\item[$\bullet$] $F \square G$ holds for {\bf all arguments} if $F(t) \square G(t)$ for all $t \geqslant 0$.
		\end{enumerate}
		}
	\end{notation}
	
	\begin{remark}
		Plainly, if both $F$ and $G$ are Young functions then $F \square G$ holds for all arguments if, and only if,
		$F \square G$ holds for small and large arguments.
	\end{remark}
	
	Let us get to the point.
	The first of the lemmas we need is well-known.
	(However, many partial results in this direction are scattered throughout the literature; see, for example, \cite[Theorem~1]{And60},
	\cite[pp.~63--68]{Dan74}, \cite[Theorems~13.7 and 13.8]{KR61}, \cite[pp.~69--75]{Mal04}, \cite[Section~VI]{ONe65} and \cite[Theorem~8]{ZR67}.)

	\begin{lemma} \label{LEMMA : 1}
		{\it Let $X$ be a Banach ideal space with the Fatou property.
		Further, let $F$, $G$ and $H$ be three Young functions.
		Suppose that one of the following three conditions holds:}
		\begin{enumerate}
			\item[(1)] {\it $F^{-1} G^{-1} \approx H^{-1}$ for all arguments and neither $L_{\infty} \hookrightarrow X$ nor $X \hookrightarrow L_{\infty}$;}
			\item[(2)] {\it $F^{-1} G^{-1} \approx H^{-1}$ for large arguments and $L_{\infty} \hookrightarrow X$;}
			\item[(3)] {\it $F^{-1} G^{-1} \approx H^{-1}$ for small arguments and $X \hookrightarrow L_{\infty}$.}
		\end{enumerate}
		{\it Then $X_F \odot X_G = X_H$.}
	\end{lemma}
	\begin{proof}
		This is just Theorems~4.1, 4.2 and 4.5 from \cite{KLM12} in tandem with Theorem~5(a) from \cite{KLM14}
		(see also \cite[Theorem~A(a)]{KLM14}; cf. \cite[Theorem~10.1, p.~69]{Mal04}).
	\end{proof}

	The {\it crux} of the entire route is proof of the following

	\begin{lemma} \label{LEMMA : 2}
		{\it Let $X$ be a rearrangement invariant function space.
		Further, let $F$, $G$ and $H$ be three Young functions.
		Suppose that $X_H \hookrightarrow X_F \odot X_G$. Then}
		\begin{itemize}
			\item[($1$)] {\it $H^{-1} \preccurlyeq F^{-1} G^{-1}$ for small arguments provided $L_{\infty} \not\hookrightarrow X$;}
			\item[($2$)] {\it $H^{-1} \preccurlyeq F^{-1} G^{-1}$ for large arguments provided $X \not\hookrightarrow L_{\infty}$;}
			\item[($3$)] {\it $H^{-1} \preccurlyeq F^{-1} G^{-1}$ for all arguments provided neither $L_{\infty} \hookrightarrow X$ nor $X \hookrightarrow L_{\infty}$.}
		\end{itemize}
	\end{lemma}
	\begin{proof}
		We will only show (1) in details.
		Once this is done, the proof of (2) is completely analogous.
		Moreover, the proof of (3) is just a simple combination of (1) and (2).
		
		Suppose that $L_{\infty} \not\hookrightarrow X$ and the condition $H^{-1} \preccurlyeq F^{-1} G^{-1}$ does not holds for small arguments.
		This means that there is a decreasing null sequence $\{ a_n \}_{n=1}^{\infty}$ of reals with
		\begin{equation} \label{EQ : a_n}
			0 < a_n < \min\left\{ 1, (F^{-1})^{-1}(b_F) \right\}
		\end{equation}
		for all $n \in \mathbb{N}$,
		where $(F^{-1})^{-1}(b_F) = \inf \{ t > 0 \colon F^{-1}(t) = b_F \}$, and
		\begin{equation} \label{EQ : brak rownowaznosci male}
			2^n F^{-1}(a_n) G^{-1}(a_n) \leqslant H^{-1}(a_n)
		\end{equation}
		for all $n \in \mathbb{N}$.
		Keeping this in mind, we claim that there is a sequence $\{ A_n \}_{n=1}^{\infty}$ of sets of positive but finite measure
		such that $\norm{a_n {\bf 1}_{A_n}}_{X} = 1$.
		To see this, just note that the function $\psi_X$ is continuous and unbounded.
		(Otherwise, $L_{\infty} \hookrightarrow X$, which is obviously not the case.)
		Thus, thanks to the Darboux property of continuous functions, we can effortlessly find $A_n$'s with $\norm{a_n {\bf 1}_{A_n}}_{X} = 1$ for $n \in \mathbb{N}$.
		Our claim follows.
		Knowing this, for $n \in \mathbb{N}$, we set
		\begin{equation*}
			f_n \coloneqq F^{-1}(a_n) {\bf 1}_{A_n}, \quad g_n \coloneqq G^{-1}(a_n) {\bf 1}_{A_n} \quad \text{ and } \quad h_n \coloneqq f_n g_n. 
		\end{equation*}
		Now, for $0 < \lambda < 1$, we have
		\begin{align*}
			\mathcal{M}_F(f_n / \lambda)
				& = \norm{F[\lambda^{-1} F^{-1}(a_n)] {\bf 1}_{A_n}}_X \\
				& \geqslant \lambda^{-1} \norm{(F \circ F^{-1})(a_n) {\bf 1}_{A_n}}_X \quad \quad (\text{due to the convexity of $F$}) \\
				& = \lambda^{-1} \norm{a_n {\bf 1}_{A_n}}_X \quad \quad (\text{thanks to \eqref{EQ : a_n}, $(F \circ F^{-1})(a_n) = a_n$}) \\
				& > 1.
		\end{align*}
		This means that $\norm{f_n}_{X_F} \geqslant 1$.
		By reasoning in exactly the same way, we can show that also $\norm{g_n}_{X_F} \geqslant 1$.
		In consequence, we have
		\begin{align*}
			\norm{h_n}_{X_F \odot X_G}
				& = F^{-1}(a_n) \, G^{-1}(a_n) \, \psi_{X_F \odot X_G}(m(A_n)) \\
				& = F^{-1}(a_n) \, \psi_{X_F}(m(A_n)) \, G^{-1}(a_n) \, \psi_{X_G}(m(A_n)) \quad \quad (\text{by \eqref{EQ : fundamental function of product}}) \\
				& = \norm{f_n}_{X_F} \norm{g_n}_{X_F} \geqslant 1.
		\end{align*}
		Thus $\norm{h_n}_{X_F \odot X_G}$ is \enquote{big}.
		Now we just need to show that $\norm{h_n}_{X_H}$ can be arbitrarily \enquote{small}.
		We have
		\begin{align*}
			\mathcal{M}_H \left( 2^n h_n \right)
				& = \norm{H \left[ 2^n F^{-1}(a_n) G^{-1}(a_n) \right] {\bf 1}_{A_n}}_X \\
				& \leqslant \norm{(H \circ H^{-1})(a_n) {\bf 1}_{A_n}}_X \quad \quad (\text{in view of \eqref{EQ : brak rownowaznosci male}}) \\
				& \leqslant \norm{a_n {\bf 1}_{A_n}}_X \quad \quad (\text{see \cite[Property~1.3]{ONe65}}) \\
				& \leqslant 1.
		\end{align*}
		In consequence, $\norm{h_n}_{X_H} \leqslant 1 / 2^n$.
		Putting this two facts together,
		\begin{equation*}
			2^n \norm{h_n}_{X_H} \leqslant \norm{h_n}_{X_F \odot X_G}
		\end{equation*}
		for $n \in \mathbb{N}$, so $X_H \not\hookrightarrow X_F \odot X_G$.
	\end{proof}

	The last lemma is quite interesting.
	This is almost literally \cite[Theorem~5(d)]{KLM14}.
	However, a very slight change in the proof resulted in drastic improvement\footnote{Strictly speaking, instead of assuming that the space
	$X$ is separable, which rules out a lot of spaces like, for example, Marcinkiewicz sequence spaces $m_{\psi}$ or Orlicz sequence spaces $\ell_F$
	without the $\Delta_2$-condition, we merely assume that $X$ is not $\ell_{\infty}$.} of the result. 
	
	\begin{lemma} \label{LEMMA : 3}
		{\it Let $X$ be a rearrangement invariant sequence space.
			Further, let $F$, $G$ and $H$ be three Young functions.
			Suppose that $\ell_{\infty} \not\hookrightarrow X$ and $X_H \hookrightarrow X_F \odot X_G$.
			Then $H^{-1} \preccurlyeq F^{-1} G^{-1}$ for small arguments.}
	\end{lemma}
	\begin{proof}
		Suppose that $\ell_{\infty} \not\hookrightarrow X$ and $H^{-1} \preccurlyeq F^{-1} G^{-1}$ does not hold for small arguments.
		This means that there is a null sequence $\{ a_n \}_{n=1}^{\infty}$ of positive integers such that
		\begin{equation*}
			2^n F^{-1}(a_n) G^{-1}(a_n) \leqslant H^{-1}(u_n)
		\end{equation*}
		for all $n \in \mathbb{N}$.
		Since $\ell_{\infty} \not\hookrightarrow X$, so $\lim_{n \rightarrow \infty} \norm{\sum_{i=1}^{n} e_n}_X = \infty$.
		(Note also that the Fatou property is important here, because otherwise $c_0$ is a simple counter-example.)
		In consequence, for every $n \in \mathbb{N}$, there exists $M(n) \in \mathbb{N}$ such that
		\begin{equation} \label{EQ : definition ciag}
			a_n \norm{ \sum_{i=1}^{M(n)} e_i }_X \leqslant 1 < a_n \norm{\sum_{i=1}^{M(n) + 1} e_i}_X.
		\end{equation}
		Since any fundamental function is quasi-concave, so $\psi_X$ is increasing, while
		\begin{equation*}
			n \mapsto \frac{\psi_X(n)}{n} = \frac{\norm{ \sum_{i=1}^{n} e_i }_X}{n}
		\end{equation*}
		is non-increasing (see \cite[Corollary~5.3, p.~67]{BS88}).
		In consequence, for $n \in \mathbb{N}$,
		\begin{equation} \label{EQ : obserwacja z tablicy}
			1 \geqslant \frac{\psi_X(n)}{\psi_X(n+1)} \geqslant \frac{n}{n+1} \geqslant \frac{1}{2}.
		\end{equation}
		Thus, combining \eqref{EQ : definition ciag} with \eqref{EQ : obserwacja z tablicy}, we see that
		\begin{equation} \label{EQ : istotne nierownosci}
			\frac{1}{2} \leqslant a_n \norm{ \sum_{i=1}^{M(n)} e_i }_X \leqslant 1.
		\end{equation}
		From this point on, the rest of the argument is the same to the letter.
		However, instead of referring to Lemma~\ref{LEMMA : 2} (or \cite[Theorem~5(d)]{KLM14}), let us finish what we started.
		Set
		\begin{equation*}
			x_n \coloneqq F^{-1}(a_n) \sum_{i=1}^{M(n)} e_i, \quad y_n \coloneqq G^{-1}(a_n) \sum_{i=1}^{M(n)} e_i \quad \text{ and } \quad z_n \coloneqq x_n y_n. 
		\end{equation*}
		Using \eqref{EQ : istotne nierownosci}, we have
		\begin{equation*}
			\mathcal{M}_F(x_n) \leqslant a_n \norm{\sum_{i=1}^{M(n)} e_i}_X \leqslant 1
		\end{equation*}
		and
		\begin{equation*}
			\mathcal{M}_F(3 x_n)
				= F\left[ 3 F^{-1}(a_n) \right] \norm{\sum_{i=1}^{M(n)} e_i}_X
				\geqslant 3 a_n \norm{\sum_{i=1}^{M(n)} e_i}_X > 1.
		\end{equation*}
		In other words, $1/3 \leqslant \norm{x_n}_{X_F} \leqslant 1$ for $n \in \mathbb{N}$.
		Repeating the same argument for $y_n$'s in place of $x_n$'s we get $1/3 \leqslant \norm{y_n}_{X_G} \leqslant 1$ for $n \in \mathbb{N}$.
		Now, it is easy to see that $\norm{z_n}_{X_F \odot X_G} \geqslant 1/9$ and $\norm{z_n}_{X_H} \leqslant 1/2^n$ for $n \in \mathbb{N}$.
		Thus, for $n \in \mathbb{N}$,
		\begin{equation*}
			2^n \norm{z_n}_{X_H} \leqslant 9 \norm{z_n}_{X_F \odot X_G}
		\end{equation*}
		and $X_H \not\hookrightarrow X_F \odot X_G$.
	\end{proof}

	\subsection{Factorization of Calder{\' o}n--Lozanovski{\u \i} spaces} \label{SUBSECTION : On factorization}
	
	Before the {\it grand finale}, we need one more definition.
	
	For a given two Banach ideal spaces $X$ and $Y$, we will say that $X$ is {\bf $Y$-perfect} if $M(M(X,Y),Y) = X$.
	Incidentally, $L_1$-perfectness of $X$ is exactly the Fatou property under disguise.
	To see this, it is enough to compare the equality
	\begin{equation*}
		M(M(X,L_1),L_1) = M(X^{\times},L_1) = X^{\times \times}
	\end{equation*}
	with the well-known fact that $X^{\times \times} = X$ if, and only if, the space $X$ has the Fatou property.
	Moreover, note that if $X$ has the Fatou property and $Y$ factorizes through $X$ then $X$ is $Y$-perfect.
	Indeed, using the so-called \enquote{cancellation} property (see, for example, \cite[Theorem~4]{KLM14}),
	we infer that 
	\begin{equation*}
		M(M(X,Y), Y)
			= M(L_{\infty} \odot M(X,Y), X \odot M(X,Y))
			= M(L_{\infty}, X)
			= X.
	\end{equation*}
	Alas, in general, there is no hope for the reverse implication to hold
	(in fact, there is a three-dimensional counter-example by Bolob{\' a}s and Brightwell; see \cite[Example~3.6]{Sch10} for a detailed presentation).
	
	Now, we are ready to show the following
	
	\begin{theorem}[Factorization of Calder{\' o}n--Lozanovski{\u \i} spaces] \label{THEOREM : Factorization of CL}
		{\it Let $X$ be a rearrangement invariant space such that $X \neq L_{\infty}$.
		Further, let $F$ and $G$ be two Young functions.
		Then the space $X_G$ can be factorized through $X_F$, that is, $X_F \odot M(X_F,X_G) = X_G$ if, and only if,
		one the following four conditions holds:}
		\begin{enumerate}
			\item[(1)] {\it $F^{-1}(G \ominus F)^{-1} \approx G^{-1}$ for all arguments and neither $L_{\infty} \hookrightarrow X$ nor $X \hookrightarrow L_{\infty}$;}
			\item[(2)] {\it $F^{-1}(G \ominus F)^{-1} \approx G^{-1}$ for large arguments and $L_{\infty} \hookrightarrow X$;}
			\item[(3)] {\it $F^{-1}(G \ominus F)^{-1} \approx G^{-1}$ for small arguments, $X \hookrightarrow L_{\infty}$ and the triple $(X,F,G)$ is nice;}
			\item[(4)] {\it $F^{-1}(G \ominus_1 F)^{-1} \approx G^{-1}$ for small arguments and either the triple $(X,F,G)$
			fails to be nice or $X$ is a sequence space.}
		\end{enumerate}
		{\it In particular, in this situation, the space $X_F$ is $X_G$-perfect.}
	\end{theorem}
	\begin{proof}
		Let $X$, $F$ and $G$ be as above.
		We will only explain how to prove $(1)$, because the rest is essentially the same.
		
		Suppose that $X_F \odot M(X_F,X_G) = X_G$.
		Since we are only interested in (1) anyway, so we can assume that neither $L_{\infty} \hookrightarrow X$ nor $X \hookrightarrow L_{\infty}$.
		Then, due to our assumption that $X \neq L_{\infty}$, so Lemma~\ref{LEMMA : 2} teaches us that $G^{-1} \preccurlyeq F^{-1} (G \ominus F)^{-1}$
		for all arguments.
		Therefore, it only remains to explain that also $F^{-1} (G \ominus F)^{-1} \preccurlyeq G^{-1}$ for all arguments.
		To see this, let us recall that the conjugate function $G \ominus F$ always satisfies generalized Young's inequality,
		that is, $G(st) \leqslant (G \ominus F)(t) + F(s)$ for $s,t \geqslant 0$ (see \eqref{DEF : generalized Young conjugate}).
		Now, it is enough to invoke Theorem~6.1 from \cite{ONe65} (see also \cite[pp.~892--893]{KLM12} and \cite[Remark~6]{KLM14}).
		
		In order to obtain the reverse implication, suppose that $F^{-1}(G \ominus F)^{-1} \approx G^{-1}$ for all arguments and neither
		$L_{\infty} \hookrightarrow X$ nor $X \hookrightarrow L_{\infty}$.
		Then, it follows from Lemma~\ref{LEMMA : 1} that $X_F \odot X_{G \ominus F} = X_G$.
		However, we also know from Theorem~\ref{THEOREM : PM between CL} that the space $X_{G \ominus F}$ coincides with $M(X_F,X_G)$.
		In consequence, the factorization $X_F \odot M(X_F,X_G) = X_G$ holds.
		
		Finally, the last part about $X_G$-perfectness of the space $X_F$ follows directly from the discussion preceding the theorem.
	\end{proof}

	\section{{\bf What's next?}} \label{SECTION : Open problems}
	
	In this closing section, let us gather and shortly discuss some problems which naturally arise from this work.
	
	\subsection{How badly can factorization fail?}
	
	The following problem (in a slightly different version) was proposed to authors by Professor Mieczys{\l}aw Masty{\l}o.
	
	\begin{problem}
		{\it Suppose that $F, G$ and $G \ominus F$ are $\mathcal{N}$-functions\footnote{Recall that an Orlicz function $F$ is called the {\bf $\mathcal{N}$-function}
		if $\lim_{t\to 0^+}\frac{F(t)}{t} = 0$ and $\lim_{t\to \infty}\frac{F(t)}{t}=\infty$ (cf. \cite[p~47]{Mal04}).}.
		Is it then true that $F^{-1}(G \ominus F)^{-1} \approx G^{-1}$?}
	\end{problem}

	This is a very intriguing question, because in every existing in the literature example showing that $F^{-1}(G \ominus F)^{-1} \not\approx G^{-1}$,
	at least on of the functions $F$ or $G$ is not an $\mathcal{N}$-function.
	A positive answer to the aforementioned question would indicate that factorization may fail only in somewhat pathological situations.
	Conversely, a negative outcome would yield new intriguing examples of Orlicz functions.
	
	\subsection{Orlicz--Lorentz spaces} \label{SUBSECTION : Orlicz--Lorentz}
	
	A function $w \in L_0$ is called the {\bf weight} whenever it is non-negative and decreasing.
	For a given weight $w$, by the {\bf weighted Orlicz space $L_F(w)$} we understand the Orlicz space associated to the Young function $F$
	and the measure $wdt$, that is,
	\begin{equation*}
		L_F(w) \coloneqq \left\{ f \in L_0 \colon \int F(\lambda \abs{f(t)}) w(t)dt < \infty \text{ for some } \lambda = \lambda(f) > 0 \right\}.
	\end{equation*}
	Recall that the {\bf Orlicz--Lorentz space} $\Lambda_F(w)$ is defined as a {\it symmetrization} of the corresponding weighted Orlicz space $L_F(w)$,
	that is, the space $\Lambda_F(w)$ consists of all $f \in L_0$ such that $f^{\star} \in L_F(w)$.
	In particular, if $F$ is just a power function, that is, $F(t) = t^p$ for some $p \geqslant 1$, the space $\Lambda_F(w)$ is usually
	denoted by $\Lambda_p(w)$ and sometimes called the {\bf Lorentz--Sharpley space}.
	
	In the early 90's, Yves Raynaud showed that the space $M(\Lambda_p(w), \Lambda_q(v))$ of pointwise multipliers between
	two Lorentz--Sharpley spaces $\Lambda_p(w)$ and $\Lambda_q(v)$ coincides, up to the equivalence of norms, with another
	Lorentz--Sharpley space $\Lambda_{r}(u)$.
	Here, $1/r = 1/q - 1/p$ and the weight $u$ verifies the relation $u^{1/r}v^{1/q} \approx w^{1/p}$ (see \cite[Proposition~25]{Ray92}).
	
	\begin{problem} \label{PROBLEM : WOL}
		{\it Provide a representation of the space $M(\Lambda_M(w), \Lambda_N(v))$ and deduce the factorization of Orlicz--Lorentz spaces.}
	\end{problem}
	
	Note that Theorem~\ref{THEOREM : PM between CL} gives an answer to the above problem only in the case when both weights are the same.
	To see this, it is enough to observe that the Orlicz--Lorentz space $\Lambda_F(w)$ coincides with the Calder{\' o}n--Lozanovski{\u \i} construction
	$(\Lambda(w))_F$, where the space $\Lambda(w)$ is defined via the norm $\norm{f}_{\Lambda(w)} \coloneqq \int f^{\star}(t)w(t)dt$.
	Then, plainly,
	\begin{equation*}
		M(\Lambda_F(w), \Lambda_G(w)) = M((\Lambda(w))_F, (\Lambda(w))_G) = (\Lambda(w))_{G \ominus F} = \Lambda_{G \ominus F}(w).
	\end{equation*}
	
	\subsection{Musielak--Orlicz setting}
	
	Another, much more general, way to look at the situation presented in this work is to consider a variant of Calder{\' o}n--Lozanovski{\u \i}'s
	construction $X_F$ in which the Young function $F$ is replaced by the so-called Musielak--Orlicz function $\Phi$ (see \cite[p.~11]{Mal04} for details).
	For lack of a better idea, and only for purposes of this paragraph, let us call them the {\bf generalized Calder{\' o}n--Lozanovski{\u \i} spaces} $X_{\Phi}$.
	In the light of the recent results obtained in \cite{LT21}, the following problem seems tempting.
	
	\begin{problem} \label{PROBLEM : MOCL}
		{\it Lift the results of this paper to the setting of generalized Calder{\' o}n--Lozanovski{\u \i} spaces.}
	\end{problem}
	
	We firmly believe that an answer to the above question is within reach, but we have not verified all the details. However we have some thoughts
	and tips for the committed reader.
	To approach this problem, we need to generalise Lemmas~3 and 4 from \cite{LT21} to the setting of generalized Calder{\' o}n--Lozanovski{\u \i} spaces.
	Then we can follow the proof of Theorem \ref{THEOREM : PM between CL} (replacing Young conjugates with the version appropriate for Musielak--Orlicz functions),
	with the main change in \enquote{Step 1} where we need the above lemmas to proceed.
	The rest is just checking the details.
	(Parenthetically speaking, we have deliberately refrained from providing any exact formulas or definitions here because they are quite
	ugly and technical. Anyway, everything - modulo references to literature - can be found in \cite{LT21}.)
	Note also that an answer to Problem~\ref{PROBLEM : MOCL} will immediately yield an answer to Problem~\ref{PROBLEM : WOL}.
	To see this, just consider Musielak--Orlicz functions of the following form $\Phi(t,s) = F(t)w(s)$, where $F$ is a Young function
	and $w$ is a weight.
	
	Actually, there is one more thing.
	
	\begin{problem}
		{\it Solve the generalized version of Problem~\ref{PROBLEM : WOL} in which the class of Orlicz--Lorentz spaces $\Lambda_F(w)$
		is replaced by the class of symmetrizations of the generalized Calder{\' o}n--Lozanovski{\u \i} spaces $X_{\Phi}$.}
	\end{problem}
	
	For now, however, this problem looks like the \enquote{ultimate horror}.
	
	\subsection{Non-symmetric variant}
	
	The most straightforward way to generalize our results is to give up the assumption about symmetry.
	
	\begin{problem}
		{\it Prove Theorems~\ref{THEOREM : PM between CL} and \ref{THEOREM : main thm not nice} without assuming that the space $X$ is rearrangement invariant.}
	\end{problem}
	
	It seems that this problem is essentially similar to the transition from the case of Orlicz spaces to Musielak--Orlicz spaces (as done in \cite{LT21}).
	Most likely, its proof will rely on the \enquote{localization} of the arguments used to show Theorems~\ref{THEOREM : PM between CL} and \ref{THEOREM : main thm not nice}.
	As before, the main obstacle will be \enquote{Step 1} in the proof of Theorem~\ref{THEOREM : PM between CL}.
	While solving this problem, however, one should not only consider the behaviour of the Young functions $F$ and $G$, but also the rate of decay
	(or lack thereof) of the norm of the indicator functions.
	
	\subsection{Non-commutative affairs}
	
	There are many papers (like, for example, those by Han \cite{Han15}, Han, Shao and Yan \cite{HSY21} or de Jager and Labuschagne \cite{JL19})
	contemplating non-commutative analogues of the results obtained in \cite{KLM12} and \cite{KLM14}.
	The prospect of enhancing them with the technology invented here seems very promising.
	All this encourages us to pose the following
	
	\begin{problem}
		{\it Provide non-commutative variants of the main results obtained here.}
	\end{problem}

\end{document}